\documentclass[11pt,a4paper,reqno]{amsart}
\usepackage{amsfonts}
\usepackage{amsmath,amssymb,amsthm,amsxtra}
\usepackage{float}
\usepackage{mathabx}
\usepackage[colorlinks, linkcolor=RoyalBlue,anchorcolor=Periwinkle,
    citecolor=Orange,urlcolor=Emerald]{hyperref}
\usepackage[usenames,dvipsnames]{xcolor}
\usepackage{enumitem}
\usepackage{geometry,array} 
\usepackage{graphicx}
\usepackage{subfigure}
\usepackage{bookmark}
\usepackage{tikz}\usetikzlibrary{matrix}
\usepackage{url}
\usepackage{dsfont}
\usepackage[colorinlistoftodos]{todonotes}

\usetikzlibrary{decorations.markings}
\tikzset{->-/.style={decoration={  markings,  mark=at position #1 with
    {\arrow{>}}},postaction={decorate}}}
\tikzset{-<-/.style={decoration={  markings,  mark=at position #1 with
    {\arrow{<}}},postaction={decorate}}}
\usepackage{extarrows}
\usepackage[all]{xy}

\theoremstyle{plain}
\newtheorem{theorem}{Theorem}[section]
\newtheorem*{thm}{Theorem}
\newtheorem{lemma}[theorem]{Lemma}
\newtheorem{corollary}[theorem]{Corollary}
\newtheorem{proposition}[theorem]{Proposition}

\theoremstyle{definition}
\newtheorem{definition}[theorem]{Definition}

\newtheorem{remark}[theorem]{Remark}

\numberwithin{equation}{section}
\newtheorem{definition-proposition}[theorem]{Definition-Proposition}
\newtheorem{construction}[theorem]{Construction}
\newtheorem{assumption}[theorem]{Assumption}

\def\surfo{\mathbf{S}_\Tri}
\def\T{\mathbf{T}}

\def\D{\mathcal{D}}

\providecommand{\OA}{\mathop{\rm OA}\nolimits}%
\providecommand{\CA}{\operatorname{CA}}%
\providecommand{\GCA}{\overline{\CA}}%
\providecommand{\Hom}{\mathop{\rm Hom}\nolimits}%

\providecommand{\Int}{\mathop{\rm Int}\nolimits}%
\providecommand{\Ext}{\mathop{\rm Ext}\nolimits}%
\providecommand{\deg}{\mathop{\rm deg}\nolimits}%
\providecommand{\per}{\mathop{\rm per}\nolimits}%
\def\hua{\mathcal}
\def\kong{\mathbb}
\def\<{\langle}
\def\>{\rangle}
\def\hom{\hua{H}om}

\def\ZZ{\mathbb{Z}}

\def\Aut{\operatorname{Aut}}

\def\Sim{\operatorname{Sim}}
\def\Hom{\operatorname{Hom}}

\def\Ext{\operatorname{Ext}}

\def\Stap{\operatorname{Stab}^\circ}

\def\deg{\operatorname{deg}}
\newcommand{\h}{\operatorname{\hua{H}}}            

\renewcommand{\k}{\mathbf{k}}

\newcommand{\tilt}[3]{{#1}^{#2}_{#3}}
\newcommand{\Cone}{\operatorname{Cone}}

\def\numbers{\begin{enumerate}[label=\arabic*{$^\circ$}.]}
\def\ends{\end{enumerate}}
\newcommand{\id}{\operatorname{id}}

\newcommand{\EG}{\operatorname{EG}}       
\newcommand{\EGp}{\operatorname{EG}^\circ}       
\newcommand{\SEGp}{\operatorname{SEG}^\circ}       
\newcommand{\C}{\hua{C}}

\newcommand\Sph{\operatorname{Sph}}
\def\zero{\hua{H}_\Gamma}
\newcommand{\Tri}{\bigtriangleup}

\def\arrow{red}
\def\surf{\mathbf{S}}                       
\def\disc{\mathbf{D}}

\newcommand{\ST}{\operatorname{ST}}        
\newcommand{\BT}{\operatorname{BT}}        

\newcommand{\MCG}{\operatorname{MCG}}

\def\Fuk{\operatorname{Fuk}}

\def\M{\mathbf{M}}
\def\P{\mathcal{P}}

\def\RHom{\operatorname{RHom}}
\def\ee{\operatorname{\mathfrak{E}}}
\def\EE{\operatorname{\hua{E}}}

\def\OAp{\operatorname{OA}^\circ}

\def\RR{\operatorname{RR}}

\newcommand\Bt[1]{\operatorname{B}_{#1}}
\newcommand\bt[1]{\operatorname{B}_{#1}^{-1}}
\def\ff{\operatorname{\mathrel{\Big|}}}

\newcommand{\Quad}{\operatorname{Quad}}

\begin{document}
\title[Decorated marked surfaces II]{Decorated marked surfaces II:
Intersection numbers and dimensions of Homs}

\author[Qiu]{Yu Qiu}
\address{YQ:
Department of Mathematics,
Chinese University of Hong Kong,
Shatin,
N.T.,
Hong Kong}
\email{yu.qiu@bath.edu}

\author[Zhou]{Yu Zhou}
\address{YZ:
Institutt for matematiske fag,
Norwegian University of Science and Technology,
N-7491,
Trondheim,
Norway}
\email{yuzhoumath@gmail.com}

\thanks{The work was supported by the Research Council of Norway, grant No.NFR:231000.}
\begin{abstract}
    We study derived categories arising from quivers with potential
    associated to a decorated marked surface $\surfo$, in the sense of \cite{QQ}.
    We prove two conjectures in \cite{QQ}, that
    under a bijection between certain objects in these categories and
    certain arcs in $\surfo$,
    the dimensions of morphisms between these objects equal
    the intersection numbers between the corresponding arcs.

\end{abstract}


\maketitle
 \section{Introduction}
\subsection{The 3-Calabi-Yau categories from surfaces}
In this paper, we study a class of derived categories $\D_{fd}(\surf)$
associated to quivers with potential from triangulated marked surfaces $\surf$.
They are 3-Calabi-Yau and originally arose in the study of homological mirror symmetry.
In type $A_n$ (or equivalently, $\surf$ an $(n+3)$-gon),
such a category was first studied by Khovanov-Seidel-Thomas \cite{KS,ST}.
They showed that there is a faithful (classical) braid group action
\[\ST\D_{fd}(\surf)\cong B_{n+1}=\MCG(\mathbf{S},n)\]
on $\D_{fd}(\surf)$,
where $\ST\D_{fd}(\surf)$ is the spherical twist group of $\D_{fd}(\surf)$ and $\MCG(\surf,n)$ the mapping class group of the disk with $n$ decorations.
This plays a crucial role in understanding such categories and
their spaces of stability conditions.
More recently,
Bridgeland-Smith \cite{BS}
established a connection between dynamical system of $\surf$
and theory of stability conditions on $\D_{fd}(\surf)$.
More precisely, they showed that
\[
    \Stap\D_{fd}(\surf)/\ST\D_{fd}(\surf)\cong\Quad(\surf),
\]
where $\Stap\D_{fd}(\surf)$ is the space of stability conditions and $\Quad(\surf)$
the moduli space of quadratic differentials.
Moreover, Smith \cite{S} showed that there is a fully faithful embedding
\[
    \D_{fd}(\surf)\hookrightarrow\D\Fuk(\mathbf{X}),
\]
where $\mathbf{X}$ is a symplectic 3-fold constructed from $\surf$
and $\D\Fuk(\mathbf{X})$ its derived Fukaya category.
This generalizes the sympelctic construction of \cite{KS}.
In the attempt of showing that $\Stap\D_{fd}(\surf)$ is the universal cover of $\Quad(\surf)$,
\cite{QQ} generalizes a result of \cite{KS} that
$\ST\D_{fd}(\surf)$ can be identified with a subgroup of the mapping class group
of $\surfo$, the decorated version of $\surf$:
\[\ST\D_{fd}(\surf)\cong\BT(\surfo)\subset\MCG(\surfo).\]

Under this embedding, the Homs between objects in $\D_{fd}(\surf)$
are in fact floer homology between the corresponding Langrangian submanifolds.
A (double graded) formula for calculating the dimension of Homs/floer homology
(in type $A$ or $\surf$ a polygon) is given in \cite{KS}.
Naturally, one would expect the corresponding formula for unpunctured $\surf$,
which is conjectured in \cite{QQ} (an ungraded version).
The main motivation of this paper is to prove such a conjecture
and another closely related formula.

\subsection{Topological aspect of cluster theory}
Cluster algebras and quiver mutation were introduced by Fomin-Zelevinsky \cite{FZ}.
Derksen-Weyman-Zelevinsky \cite{DWZ} further developed quiver mutation to mutation of quivers with potential. During the last decade,
the cluster phenomenon was spotted in various areas in mathematics, as well as in physics,
including geometric topology and representation theory.

On one hand, the geometric aspect of cluster theory was explored by
Fomin-Shapiro-Thurston \cite{FST}.
They constructed a quiver $Q_\T$ (and later Labardini-Fragoso \cite{LF} gave a corresponding potential $W_\T$)
from any (tagged) triangulation $\T$ of a marked surface $\surf$.
Moreover, they showed that mutation of quivers (with potential)
is compatible with flip of triangulations.
On the other hand, the categorification of cluster algebras
leads to representations of quivers,
due to Buan-Marsh-Reineke-Reiten-Todorov \cite{BMRRT}.
Later, Amiot \cite{A} introduced generalized cluster categories via Ginzburg dg algebras associated to quivers with potential, that
$\D_{fd}(\surf)$ fits into the following short exact sequence of triangulated categories:
\begin{gather}
    0 \to \D_{fd}(\surf) \to \per\surf  \to \C(\surf) \to 0
\end{gather}
where $\C(\surf)$ is the generalized cluster category.

Several works have been done concerning these categories associated to marked surfaces.
Namely, the following correspondences were established
(cf. Table~\ref{table}):
\begin{itemize}
  \item In the unpunctured case,
  Br\"{u}stle-Zhang \cite{BZ} constructed a bijection
  between the set of open arcs on $\surf$ and the set of indecomposables in the cluster category $\C(\surf)$.
  \item Qiu-Zhou \cite{QZ} constructed such a bijection in the general case
  (i.e. with punctures).
  \item Qiu \cite{QQ} constructed a bijection between the set of (simple) closed arcs on
  $\surfo$ (the decorated version of $\surf$) and the set of shift orbits of
  (reachable) spherical objects in $\D_{fd}(\surf)$.
  \item Qiu \cite{QQ2} constructed a bijection between the set of (simple) open arcs on
  $\surfo$ and the set (reachable) rigid indecomposable objects in $\D_{fd}(\surf)$.
\end{itemize}

\begin{table}[ht]
\caption{Topological model for categories associated to quivers with potential}
\label{table}
\setlength{\extrarowheight}{2pt}
\begin{tabular}{|c|l|c|}\hline
Topological model&Correspondence&Categories\\[1ex]
\hline
$\surfo$&& $\D_{fd}(\surf)$: 3-CY\\
Close arcs& $\widetilde{X}\colon\eta\xrightarrow{\;\text{\cite{QQ}}\;} X_\eta[\kong{Z}]\qquad$ &reachable spherical objects
\\\hline
$\surfo$&&$\per\surf$\\
Open arcs& $\widetilde{M}\colon\gamma\xrightarrow{\;\text{\cite{QQ2}}\;} \widetilde{M}_\gamma\qquad$ &reachable ind. objects
\\\hline
$\surf$&&$\C(\surf)$: 3-CY\\
Open arcs& $M\colon\mu\xrightarrow{\text{\cite{BZ,QZ}}} M_{\mu}\qquad$ &ind. objects
\\\hline
\end{tabular}
\end{table}

Furthermore, there are several the $\Int=\dim\Hom$ type formulae under this type of correspondences:
\numbers
  \item Khovanov-Seidel \cite{KS} showed
  $\dim\Hom^\ZZ(\widetilde{X}_{\eta_1},\widetilde{X}_{\eta_2})=2\Int(\eta_1,\eta_2)$
  in the case of $\surf$ of type $A$.
  \item Gadbled-Thiel-Wagner \cite{GTW} showed that the formula in $1^\circ$ holds for $\surf$ of extended affine type $A$.
  \item Zhang-Zhu-Zhou \cite{ZhZZ} showed that
  $\dim\Hom^1(M_{\gamma_1},M_{\gamma_2})=\Int(\gamma_1,\gamma_2)$
  for $\surf$ unpunctured.
  \item Qiu-Zhou \cite{QZ} showed that the formula in $3^\circ$ holds for the punctured case.
\ends
Here $\Hom^\ZZ(-,-)$ denotes $\oplus_{m\in\mathbb{Z}}\Hom(-,-[m])$,
which can be defined for shift orbits of objects.
Note that the differences between $\Hom^\ZZ$ and $\Hom^1$ reflect
the different properties that $\D_{fd}(\surf)$ is 3-Calabi-Yau
while $\C(\surf)$ is 2-Calabi-Yau.
The main technics for proving these $\Int=\dim\Hom$ type formulae
are the string model (or its variation/generalization) in representation theory.
In this paper, we establish some frameworks of graded string model for
certain 3-Calabi-Yau categories and
prove the following two formulae (Theorem~\ref{thm:1},\ref{thm:2}) of this type.

\begin{thm}\cite[Conjecture~10.5,10.6]{QQ}
Under the correspondence $\widetilde{M},\widetilde{X}$ in Table~\ref{table},
the following formulae hold:
\begin{gather*}
\dim\Hom^\ZZ(\widetilde{X}_{\eta_1},\widetilde{X}_{\eta_2})=2\Int(\eta_1,\eta_2),\\
\dim\Hom^\ZZ(\widetilde{M}_\gamma,\widetilde{X}_{\eta})=\Int(\gamma,\eta).
\end{gather*}
\end{thm}

\subsection{Context}
The paper is organized as follows.
In Section~\ref{Sec:2}, we review background materials.
In Section~\ref{Sec:3}, we prove a first formula for spherical objects under Assumption~\ref{ass}.
In Section~\ref{Sec:4}, we show that one can identify all sets of reachable spherical objects
from different triangulations in a canonical way.
This enables us to generalize the first formula to all cases and we also prove
a second formula as a byproduct.
In Appendix~\ref{app:A}, we develop the graded string model,
which is independent with the rest of the paper.
The key result here is the calculations of a type of
morphisms between (spherical) objects and the compositions of these morphisms.
This appendix also serves as a technical section for the prequel \cite{QQ}.
\subsection*{Acknowledgements}
The authors would like to thank Alastair King, Tom Bridgeland, Ivan Smith and Dong Yang for
interesting discussions. The second author would like to thank Henning Krause
and his research group for a pleasant atmosphere
when he was a postdoc at the Faculty of Mathematics of Bielefeld University. We would like to thank the anonymous referee for many concrete suggestions for improving the paper.

\section{Preliminaries}\label{Sec:2}

\subsection{Triangulated 3-Calabi-Yau categories and spherical twists}

Fix an algebraically closed filed $\k$ and all categories are $\k$-linear. A triangulated category $\mathcal{T}$ is called 3-Calabi-Yau if for any objects $X,Y\in\mathcal{T}$, there is a functorial isomorphism
\[\Hom_{\mathcal{T}}(X,Y)\cong D\Hom_{\mathcal{T}}(Y,X[3])\]
where $D=\Hom_\k(-,\k)$ the $\k$-duality. An (indecomposable) object $S$ in a triangulated 3-Calabi-Yau category $\mathcal{T}$ is called ($3$-)spherical if $\Hom_{\mathcal{T}}(S,S[n])$ equals $\k$ if $n=0$ or $3$ and equals zero otherwise. Recall from \cite{ST} that
the twist functor of a spherical object $S$ is defined by
\[\phi_S(X)=\Cone\left(S\otimes\Hom^\ZZ_\mathcal{T}(S,X)\rightarrow X\right)\]
with inverse
\[\phi^{-1}_S(X)=\Cone\left(X\rightarrow S\otimes\Hom^\ZZ_\mathcal{T}(X,S)\right)[-1].\]

\subsection{Quivers with potential, Ginzburg dg algebras and associated categories}\label{subsec:dgquiver}

A quiver with potential \cite{DWZ} is a pair $(Q,W)$,
where $Q$ is a finite quiver and $W$ is a linear combination of cycles in $Q$. The Ginzburg dg algebra \cite{G} $\Gamma=\Gamma(Q,W)$ associated to $(Q,W)$ is defined as follows. Let $\overline{Q}$ be the graded quiver with the same set of vertices as $Q$ and whose arrows are:
\begin{itemize}
    \item the arrows of $Q$ with degree 0,
    \item an arrow $a^\ast:j\rightarrow i$ with degree $-1$ for each arrow $a:i\rightarrow j$ of $Q$,
    \item a loop $t_i:i\rightarrow i$ with degree $-2$ for each vertex $i$ of $Q$.
\end{itemize}
The underlying graded algebra of $\Gamma$ is the completion of the graded path algebra $\k\overline{Q}$ and the differential of $\Gamma$ is determined uniquely by the following
\begin{itemize}
    \item $d(a)=0$ and $d(a^\ast)=\partial_aW$, for $a$ an arrow of $Q$,
    \item $\sum_{i\in Q_0} d(t_i)=\sum_{a\in Q_1}[a,a^*]$.
\end{itemize}
Let $\D(\Gamma)$ be the derived category of $\Gamma$. We consider the following full subcategories of $\D(\Gamma)$:
\begin{itemize}
  \item $\per\Gamma$: the perfect derived category of $\Gamma$,
  \item $\D_{fd}(\Gamma)$: the finite dimensional derived category of $\Gamma$.
\end{itemize}
It is known that $\per\Gamma$ is Krull-Schmidt \cite{KY} and $\D_{fd}(\Gamma)$ is 3-Calabi-Yau \cite{Ke}.
Let $\zero$ be the canonical heart of $\D_{fd}(\Gamma)$ and $\Sim\zero$ be the set of its simples.
As in \cite{QQ}, we use the following notations:
\begin{itemize}
  \item $\ST(\Gamma)$: the spherical twist group, which is the subgroup of $\Aut\D_{fd}(\Gamma)$ generated by $\phi_{S}$ for $S\in\Sim\zero$,
  \item $\Sph(\Gamma)$: the set of reachable spherical objects in $\D_{fd}(\Gamma)$,
  i.e. $\ST(\Gamma)\cdot\Sim\zero$.
\end{itemize}

\subsection{Triangulations of marked surfaces}

An (unpunctured) marked surface $\surf$ is an oriented compact surface with a finite set $\M$ of marked points lying on its non-empty boundary $\partial\surf$ \cite{FST}. Up to homeomorphism, a marked surface $\surf$ is determined by the following data:
\begin{itemize}
  \item the genus $g$ of $\surf$;
  \item the number $b$ of components of $\partial\surf$;
  \item the partition of the number $m=|\M|$
  describing the numbers of marked points on components of $\partial\surf$.
\end{itemize}
An (open) arc $\gamma$ in $\surf$ is a curve in the surface satisfying
\begin{itemize}
  \item the endpoints of $\gamma$ are in $\M$;
  \item except for its endpoints, $\gamma$ is disjoint from $\partial\surf$;
  \item $\gamma$ has no self-intersections in $\surf-\M$;
  \item $\gamma$ is not isotopic to a point or a boundary segment.
\end{itemize}
The arcs are always considered up to isotopy. A triangulation $\mathbb{T}$ of $\surf$ is a maximal collection of arcs in $\surf$ which do not intersect each other in the interior of $\surf$. We have
\[n:=|\mathbb{T}|=6g+3b+m-6,\]
and the number $\aleph$ of the triangles in $\mathbb{T}$ is
$(2n+m)/3$.

There is a quiver with potential $(Q_\mathbb{T},W_\mathbb{T})$ \cite{FST,LF},
associated to each triangulation $\mathbb{T}$ of $\surf$ as follows:
\begin{itemize}
  \item the vertices of $Q_\mathbb{T}$ are indexed by the arcs in $\mathbb{T}$;
  \item there is an arrow $i\to j$ whenever $i$ and $j$ are edges of the same triangle and $j$ follows $i$ clockwise, hence each triangle with three edges in $\mathbb{T}$ gives a 3-cycle (up to cyclic permutation);
  \item the potential $W_\mathbb{T}$ is the sum of the such 3-cycles.
\end{itemize}

\subsection{Decorated marked surfaces}
A decorated marked surface $\surfo$ is a marked surface with an extra set $\bigtriangleup$ of $\aleph$ decorating points in the interior of $\surf$.
A general closed arc $\eta$ in $\surfo$ is a curve in $\surf$ such that
\begin{itemize}
  \item its endpoints are in $\triangle$;
  \item except for its endpoints, $\gamma$ is disjoint from $\bigtriangleup$ and from $\partial\surf$;
  \item it is not isotopic to a point.
\end{itemize}
A closed arc in $\surfo$ is a general closed arc whose endpoints do not coincide.
An open arc $\gamma$ in $\surfo$ is a curve in $\surf$ such that
\begin{itemize}
  \item its endpoints are in $\M$;
  \item except for its endpoints, $\gamma$ is disjoint from $\bigtriangleup$ and from $\partial\surf$;
  \item it is not isotopic to a point or a boundary component.
\end{itemize}
We denote by $\CA(\surfo)$, $\GCA(\surfo)$ and $\OA(\surfo)$
the set of simple closed, simple general closed and simple open arcs in $\surfo$, respectively. Recall from \cite[\S~3.1]{QQ} the notion of intersection numbers as follows.
\begin{itemize}
  \item For an open arc $\gamma$ and an (open or general closed) arc $\eta$, their intersection number is defined as the geometric intersection number in $\surf_\triangle-\M$:
  \[\Int(\gamma,\eta)=\min\{|\gamma'\cap\eta'\cap(\surf_\triangle-\M)|\mid\gamma'\sim\gamma,\eta'\sim\eta\}.\]
  \item For two general closed arcs $\alpha,\beta$ in $\GCA(\surfo)$,
their intersection number is a half integer in $\tfrac{1}{2}\ZZ$ and defined as follows
(following \cite{KS}):
\[
    \Int(\alpha,\beta)
    =\tfrac{1}{2}\Int_{\Tri}(\alpha,\beta)+\Int_{\surf-\Tri}(\alpha,\beta),
\]
where
\[
    \Int_{\surf-\Tri}(\alpha,\beta)=\min\{ |\alpha'\cap\beta'\cap (\surf-\Tri)|
        \ff \alpha'\sim\alpha,\beta'\sim\beta \}
\]
and $$\Int_{\Tri}(\alpha,\beta)=\sum_{Z\in\Tri} |\{t\mid\alpha(t)=Z\}|\cdot|\{r\mid\beta(r)=Z\}|.$$
\end{itemize}
A triangulation $\T$ is a maximal collection of open arcs in $\surfo$ such that
\begin{itemize}
  \item for any $\gamma_1,\gamma_2\in\T$, $\Int(\gamma_1,\gamma_2)=0$;
  \item each triangle of $\T$ contains exactly one (decorating) point in $\triangle$.
\end{itemize}
The forgetful map $F:\surfo\rightarrow \surf$ forgetting the decorating points induces a map from $\OA(\surfo)$ to the set of arcs in $\surf$, which sends a triangulation $\T$ of $\surfo$ to a triangulation $\mathbb{T}=F(\T)$ of $\surf$. The quiver with potential $(Q_\T,W_\T)$ associated to $\T$ is defined to be $(Q_\mathbb{T},W_\mathbb{T})$.

Let $\gamma$ be an (open) arc in a triangulation $\T$. The arc $\gamma^\sharp=\gamma^\sharp(\T)$ (resp. $\gamma^\flat$) is the arc obtained from $\gamma$ by anticlockwise (resp. clockwise) moving its endpoints along the quadrilateral in $\T$ whose diagonal is $\gamma$, to the next marked points. The forward (resp. backward) flip of a triangulation $\T$ at $\gamma\in\T$ is the triangulation $\T_\gamma^\sharp=\T\cup\{\gamma^\sharp\}-\{\gamma\}$ (resp. $\T_\gamma^\flat=\T\cup\{\gamma^\flat\}-\{\gamma\}$).
See Figure~\ref{fig:5+} for example.
The exchange graph $\EG(\surfo)$ is the oriented graph whose vertices are triangulations of $\surfo$ and whose arrows correspond to forward and backward flips. From now on, fix a connected component $\EGp(\surfo)$ of $\EG(\surfo)$. When we say a triangulation $\T$ of $\surfo$, we mean $\T$ is in $\EGp(\surfo)$.
\begin{figure}[ht]\centering
\begin{tikzpicture}[xscale=.8,yscale=1]
  \foreach \j in {0, 8}
  {
    \draw[thick](0+\j,0)node[NavyBlue]{$\bullet$}to(2+\j,0)node[NavyBlue]{$\bullet$}to(3+\j,2)node[NavyBlue]{$\bullet$}
    to(1+\j,3)node[NavyBlue]{$\bullet$}to(-1+\j,2)node[NavyBlue]{$\bullet$}to(0+\j,0);
    \draw(1+\j,1)node[red]{$\circ$}(\j,1.666)node[red]{$\circ$}(2+\j,1.666)node[red]{$\circ$};
  }
\draw[NavyBlue,thick](0,0)to(1,3)to(2,0) (0,4.5)node{};
\draw[NavyBlue,thick](-1+8,2).. controls +(5:3) and +(170:3) ..(2+8,0)to(1+8,3);
\draw[blue,->,>=stealth] (1+4-1.5,1.5)to(1+4+1.5,1.5);
\end{tikzpicture}
\caption{A forward flip of a triangulation}
\label{fig:5+}
\end{figure}
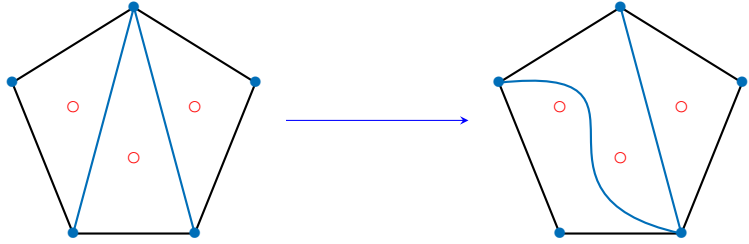
\subsection{The braid twists}\label{sec:twist}

The mapping class group $\MCG(\surfo)$ of $\surfo$ consists of the isotopy classes of the homeomorphisms of $\surf$ that fix $\partial\surf$ pointwise and fix the set $\triangle$.

For any closed arc $\eta\in\CA(\surfo)$, the braid twist $B_\eta\in\MCG(\surfo)$ along $\eta$ is defined as in Figure~\ref{fig:Braid twist}.
The braid twist group $\BT(\surfo)$ is defined as the subgroup of $\MCG(\surfo)$ generated by $B_\eta$ for $\eta\in\CA(\surfo)$.
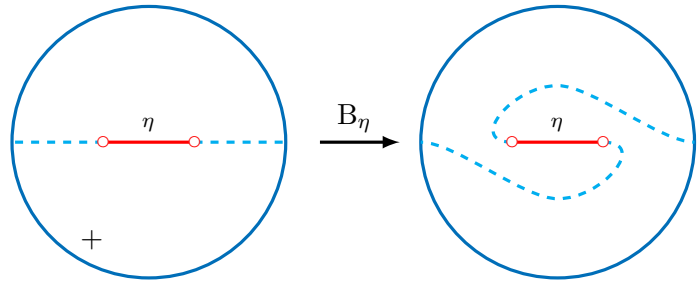
\begin{figure}[ht]\centering
\begin{tikzpicture}[scale=.3]
  \draw[very thick,NavyBlue](0,0)circle(6)node[above,black]{$_\eta$};
  \draw(-120:5)node{+};
  \draw(-2,0)edge[red, very thick](2,0)  edge[cyan,very thick, dashed](-6,0);
  \draw(2,0)edge[cyan,very thick,dashed](6,0);
  \draw(-2,0)node[white] {$\bullet$} node[red] {$\circ$};
  \draw(2,0)node[white] {$\bullet$} node[red] {$\circ$};
  \draw(0:7.5)edge[very thick,->,>=latex](0:11);\draw(0:9)node[above]{$\Bt{\eta}$};
\end{tikzpicture}\;
\begin{tikzpicture}[scale=.3]
  \draw[very thick, NavyBlue](0,0)circle(6)node[above,black]{$_\eta$};
  \draw[red, very thick](-2,0)to(2,0);
  \draw[cyan,very thick, dashed](2,0).. controls +(0:2) and +(0:2) ..(0,-2.5)
    .. controls +(180:1.5) and +(0:1.5) ..(-6,0);
  \draw[cyan,very thick,dashed](-2,0).. controls +(180:2) and +(180:2) ..(0,2.5)
    .. controls +(0:1.5) and +(180:1.5) ..(6,0);
  \draw(-2,0)node[white] {$\bullet$} node[red] {$\circ$};
  \draw(2,0)node[white] {$\bullet$} node[red] {$\circ$};
\end{tikzpicture}
\caption{The braid twist}
\label{fig:Braid twist}
\end{figure}
For a triangulation $\T=\{\gamma_1,\ldots,\gamma_n\}$ of $\surfo$, its dual triangulation $\T^\ast$ consists of the closed arcs $s_1,\ldots, s_n$ in $\CA(\surfo)$ satisfying that $\Int(\gamma_i,s_j)$ equals 1 for $i=j$ and equals 0 otherwise.

\subsection{Topological preparation}
We start with a lemma.
Let
$$\disc=\{(x,y)\in\mathds{R}^2\mid x^2+y^2<1\}$$
be a disk with three punctures $P_0=(0,\frac{1}{2})$, $P_1=(0,-\frac{1}{2})$, $P_2=(0,0)$.
Set \[\disc^{>0}=\{(x,y)\in\disc\mid y>0\}\quad\text{and}\quad\disc^{<0}=\{(x,y)\in\disc\mid y<0\}.\]
In this subsection, when we mention a curve, we always mean a continuous map from $[0,1]$ to $\disc$ such that it is disjoint with the punctures except for its endpoints. 
Let $\eta:[0,1]\rightarrow \disc$ be a curve.
Denote by $\overline{\eta}$ the curve defined as $\overline{\eta}(t)=\eta(1-t)$.
The restriction $\eta|_{[t_1,t_2]}$ is a curve defined as $\eta|_{[t_1,t_2]}(t)=\eta\left((t_2-t_1)t+t_1\right)$ if $t_1<t_2$,
and $\eta|_{[t_1,t_2]}=\overline{\eta|_{[t_2,t_1]}}$ if $t_1>t_2$.
For any two curves $\eta_1,\eta_2$ with $\eta_1(1)=\eta_2(0)\notin\{P_1,P_2,P_3\}$, their composition $\eta_2\eta_1$ is a curve defined by
$\eta_2\eta_1(t)=\eta_1(2t)$ for $0\leq t\leq \frac{1}{2}$ and
$\eta_2\eta_1(t)=\eta_2(2t-1)$ for $\frac{1}{2}\leq t\leq 1$.
For a simple curve $\eta$ whose endpoints coincide, denote by $D_\eta$ the disk (possibly with punctures) enclosed by $\eta$.

\begin{lemma}\label{lem:homotopy}
Let $\eta:[0,1]\rightarrow\disc$ be a simple curve with $\eta(0)=P_0$ and $\eta(1)=P_1$. Assume that $\eta$ is in a minimal position w.r.t. $Y=\disc\cap\{y=0\}$, with $\Int(\eta,Y)>2$.
Then there exist two simple curves
$\eta_0\subset\disc^{>0}$ and $\eta_1\subset\disc^{<0}$
satisfying
\begin{itemize}
\item $\eta_0(0)=P_0$, $\eta_0(1)=\eta(s_0)$, $\eta_1(0)=\eta(s_1)$ and $\eta_1(1)=P_1$ for some $0<s_i<1$;
\item $\eta_0\nsim\eta|_{[0,s_0]}$ and $\eta_1\nsim\eta|_{[s_1,1]}$;
\item the curves $\eta_0$, $\eta_1$ intersect $\eta$ at $\eta(s_0)$, $\eta(s_1)$ respectively from different sides;
\item the curve $\alpha_0:=\eta|_{[s_0,1]}\eta_0$ is isotopic to $\alpha_1:=\eta_1\eta|_{[0,s_1]}$ relative to $\{0,1\}$.
\end{itemize}
\end{lemma}
\begin{proof}
Let $\eta\cap Y=\{\eta(r_i)\mid 0<r_1<\cdots<r_m<1\}$.
As $m>2$, we can connect $P_0$ to a point in a
segment $\eta_{[r_i,r_{i+1}]}\subset\disc^{>0}$
for some $i>0$
without intersecting $\eta$ except for the endpoints
to get an arc $\eta_0$. Similarly, we can get an arc $\eta_1$.
Moreover, $\eta_0\nsim\eta|_{[0,s_0]}$ and $\eta_1\nsim\eta|_{[s_1,1]}$
since $\eta$ is in a minimal position w.r.t. $Y$.
Let $c_0=\overline{\eta_0}\eta|_{[0,s_0]}$ and $c_1=\eta|_{[s_1,1]}\overline{\eta_1}$.
It follows that the corresponding disks $D_{c_i}$ are not contractible and hence contain at least one puncture.
Now we claim that $D_{c_0}\subset D_{c_1}$ or $D_{c_1}\subset D_{c_0}$.
Otherwise, they are disjoint since $c_0$ and $c_1$ do not intersect transversely.
So $D_{c_i}$ does not contain $P_0,P_1$;
hence both $D_{c_0}$ and $D_{c_1}$ have to contain $P_2$.
This contradicts the fact that they are disjoint.

Without loss of generality, we assume that $D_{c_0}\subset D_{c_1}$ and $\eta_0$ intersects $\eta$ at $\eta(s_0)$ from the left side.
Then up to isotopy, there are three cases shown in Figure~\ref{fig:proof}.
In the first two cases, $\eta_1$ intersects $\eta$ at $\eta(s_1)$ from the right side and the disk $D_{\overline{\alpha_1}\alpha_0}$ contains no punctures.
Hence $\alpha_0\sim\alpha_1$ relative to $\{0,1\}$ and we are done.
In case (c), we have $\eta\sim\eta_1\eta|_{[s_0,s_1]}\eta_0$.
But $\eta_1\eta|_{[s_0,s_1]}\eta_0$ has less intersections with $Y$ than $\eta$,
which is a contradiction.
This completes the proof.
\end{proof}

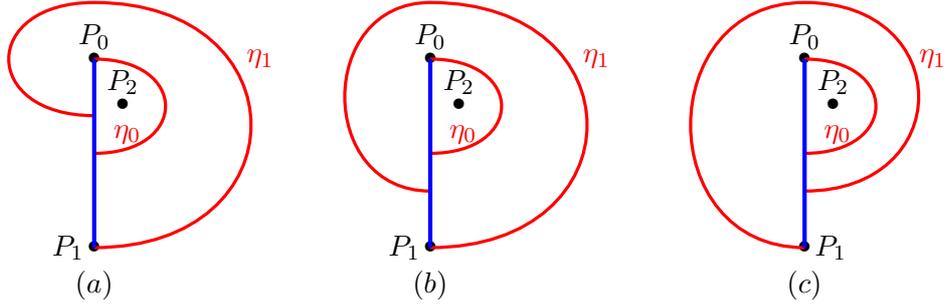
\begin{figure}[t]\centering
\begin{tikzpicture}[scale=.5]
\draw (90:2.5)node{$\bullet$} (-90:2.5)node{$\bullet$} (60:1.5)node{$\bullet$};
\draw (90:2.5)node[above]{$P_0$} (-90:2.5)node[left]{$P_1$} (60:1.5)node[above]{$P_2$};
\draw[red, very thick] (90:2.5) .. controls +(0:2.5) and +(0:2.5)..(90:0);
\draw[red, very thick] (90:-2.5) .. controls +(0:5.5) and +(0:5.5)..  (90:4) .. controls +(-180:3) and +(-180:3) .. (90:1);
\draw[red] (5,2.5)node[left]{$\eta_1$} (30:1)node{$\eta_0$};
\draw (-90:3.5)node{$(a)$};
\draw[blue,ultra thick] (90:2.5) to (-90:2.5);
\end{tikzpicture}
\begin{tikzpicture}[scale=.5]
\draw (90:2.5)node{$\bullet$} (-90:2.5)node{$\bullet$} (60:1.5)node{$\bullet$};
\draw (90:2.5)node[above]{$P_0$} (-90:2.5)node[left]{$P_1$} (60:1.5)node[above]{$P_2$};
\draw[red, very thick] (90:2.5) .. controls +(0:2.5) and +(0:2.5)..(90:0);
\draw[red, very thick] (90:-2.5) .. controls +(0:5.5) and +(0:5.5)..  (90:4) .. controls +(-180:3) and +(-180:3) .. (90:-1);
\draw[red] (5,2.5)node[left]{$\eta_1$} (30:1)node{$\eta_0$};
\draw (-90:3.5)node{$(b)$};
\draw[blue,ultra thick] (90:2.5) to (-90:2.5);
\end{tikzpicture}
\begin{tikzpicture}[scale=.5]
\draw (90:2.5)node{$\bullet$} (-90:2.5)node{$\bullet$} (60:1.5)node{$\bullet$};
\draw (90:2.5)node[above]{$P_0$} (-90:2.5)node[right]{$P_1$} (60:1.5)node[above]{$P_2$};
\draw[red, very thick] (90:2.5) .. controls +(0:2.5) and +(0:2.5)..(90:0);
\draw[red, very thick] (90:-1) .. controls +(0:4) and +(0:4)..  (90:4) .. controls +(-180:4) and +(-180:4) .. (90:-2.5);
\draw[red] (4,2.5)node[left]{$\eta_1$} (30:1)node{$\eta_0$};
\draw (-90:3.5)node{$(c)$};
\draw[blue,ultra thick] (90:2.5) to (-90:2.5);
\end{tikzpicture}
\caption{The three cases in the proof of Lemma~\ref{lem:homotopy} (topological view)}
\label{fig:proof}
\end{figure}

Now we generalize \cite[Lemma~3.14]{QQ} to the case that $\T$ is an arbitrary triangulation of $\surfo$. Recall that $\T^\ast$ is the dual triangulation of $\T$.

\begin{lemma}\label{lem:dcp}
Let $\eta$ be a closed arc in $\CA(\surfo)$ which is not in $\T^\ast$. Then there are two arcs $\alpha$, $\beta$ in $\CA(\surfo)$ such that
\numbers
\item $\Int_{\surf-\Tri}(\alpha,\beta)=0 \quad(\text{so }\Int(\alpha,\beta)\leq1)$,
\item $\eta=B_\alpha(\beta)\ \text{or}\ \eta=B_\alpha^{-1}(\beta)$,
\item $\Int(\gamma_i,\alpha)<\Int(\gamma_i,\eta)$ and
$\Int(\gamma_i,\beta)<\Int(\gamma_i,\eta)$ for any $\gamma_i\in\T$.
\ends
\end{lemma}

\begin{proof}
Assume that $\eta$ has minimal intersections with the arcs in $\T$
without loss of generality.
If $\eta$ intersects at least three triangles of $\T$, then the assertion holds by the proof of \cite[Lemma~3.14]{QQ}. Thus we can suppose that $\eta$ intersects exactly two triangles of $\T$.

Since the original marked surface $\surf$ is not a once-punctured torus, these two triangles that intersect $\eta$
cannot share three edges. On the other hand, if they share only one edge, say $i$, then $\eta=s_i$ because $\eta$ is contained in these two triangles. This contradicts one assumption. Therefore, these two triangles share exactly two edges and they form an annulus $A$.
As we only care the interior of the union of these two triangles, we are in the situation of Lemma~\ref{lem:homotopy}:
\begin{itemize}
\item the two boundaries of $A$ correspond to $\partial\disc$ and the puncture $P_2$, respectively;
\item the endpoints of $\eta$ correspond to punctures $P_0$ and $P_1$, respectively;
\item the sharing edges topologically correspond to $Y$.
\end{itemize}
Now, since $\Int(\eta,\T)>2$, there exist arcs $\eta_0$ and $\eta_1$
satisfying the conditions in Lemma~\ref{lem:homotopy}.
Let $\alpha=\alpha_1\sim\alpha_2$ and
$\beta=\eta_1\eta|_{[s_0,s_1]}\eta_0$.
Then $\Int_{\surf-\Tri}(\alpha,\beta)=0$,
$\Int_{\Tri}(\alpha,\beta)=2$
and $\eta=B_\alpha(\beta)$
(for the case that $\eta_0$ intersects $\eta$
at $\eta(s_0)$ from the left side)
or $\eta=B_\alpha^{-1}(\beta)$
(for the case that $\eta_0$ intersects $\eta$
at $\eta(s_0)$ from the right side).
Note that $\eta_0$ and $\eta_1$ do not cross any arcs in $\T$.
Then for any $\gamma_i\in\T$,
$$\Int(\gamma_i,\alpha)=\Int(\gamma_i,\eta|_{[0,s_1]})=\Int(\gamma_i,\eta|_{[s_0,1]})<\Int(\gamma_i,\eta)$$ and $$\Int(\gamma_i,\beta)=\Int(\gamma_i,\eta|_{[s_0,s_1]})<\Int(\gamma_i,\eta).$$
Thus, we complete the proof.
\end{proof}

As an immediate consequence, the proof in \cite[Proposition~4.3 and Proposition~4.4]{QQ}
works for all cases (i.e. without Assumption~\ref{ass}):

\begin{proposition}\cite[Proposition~4.3 and Proposition~4.4]{QQ}\label{prop:r}
For any triangulation $\T$ of $\surfo$, we have
$$\BT(\surfo)=\BT(\T)\quad\text{and}\quad\CA(\surfo)=\BT(\surfo)\cdot\T^\ast.$$
\end{proposition}

\section{Intersection numbers and dimensions of Homs}\label{Sec:3}
Recall that in \cite{QQ}, the author gives a bijection
from the set of closed arcs in $\surfo$ to the set of reachable spherical objects in $\D_{fd}(\Gamma_{\T_0})$,
for any triangulation $\T_0$ such that any two of its triangles share at most one edge.
In this section, we generalize this bijection to arbitrary triangulation $\T$ (of any decorated marked surface).

\subsection{The string model}\label{sec:above}
As in Appendix~\ref{app:A}, we have the following.
For each $\eta\in\GCA(\surfo)$, $X_\eta$ is an object in $\D_{fd}(\Gamma_{\T})$,
which induces a map
\begin{gather}\label{eq:X}
    \widetilde{X}_{\T}\colon \CA(\surfo)\to \D_{fd}(\Gamma_{\T})/[1],
\end{gather}
\[\quad\qquad\eta\mapsto X_\eta[\ZZ].\]
The notation $X[\ZZ]$ means the shift orbit $\{X[i]\mid i\in\mathbb{Z}\}$.
Moreover,  
Let $\sigma,\tau$ be oriented general closed arcs in $\surfo$ with $\Int_{\surf-\Tri}(\sigma,\tau)=0$ and whose starting points coincide. 
Proposition~\ref{defpp} gives a non-zero morphism 
$$\varphi(\sigma,\tau)\in\Hom_{\D_{fd}(\Gamma_{\T})}(X_\sigma,X_\tau[\upsilon]).$$
In the following, we keep the notations in Appendix~\ref{app:A}
and upgrade Proposition~\ref{prop:key} first.

\begin{proposition}\label{pp}
If $\sigma,\tau$ only share their starting but not ending points,
then there is a non-split triangle in $\D_{fd}(\Gamma_{\T})$, whose image in $\D_{fd}(\Gamma_{\T})/[1]$ is
\begin{gather}\label{eq:SES1}
\widetilde{X}(B_\sigma(\tau))\longrightarrow
\widetilde{X}(\sigma)\xrightarrow{\varphi(\sigma,\tau)} \widetilde{X}(\tau)\longrightarrow \widetilde{X}(B_\sigma(\tau))
\end{gather}
If $\sigma,\tau$ share both the endpoints, i.e. $\Int_{\bigtriangleup}(\sigma,\tau)=2$, then there is a non-split triangle in $\D_{fd}(\Gamma_{\T})$, whose image in $\D_{fd}(\Gamma_{\T})/[1]$ is
\begin{gather}\label{eq:SES}
\widetilde{X}(B_\sigma(\tau))\longrightarrow
\widetilde{X}(\sigma) \oplus \widetilde{X}(\sigma)\xrightarrow{\left(\begin{smallmatrix}
\varphi(\sigma,\tau)&\varphi(\overline{\sigma},\overline{\tau})
\end{smallmatrix}\right)} \widetilde{X}(\tau)\longrightarrow \widetilde{X}(B_\sigma(\tau))
\end{gather}
where $\varphi(\sigma,\tau)$ and $\varphi(\overline{\sigma},\overline{\tau})$ are linearly independent.
\end{proposition}

\begin{proof}
For the case $\Int_{\bigtriangleup}(\sigma,\tau)=1$, we have $B_\sigma(\tau)=\tau\wedge\sigma$. So the triangle in Prop~\ref{prop:key} becomes \eqref{eq:SES1}.
	
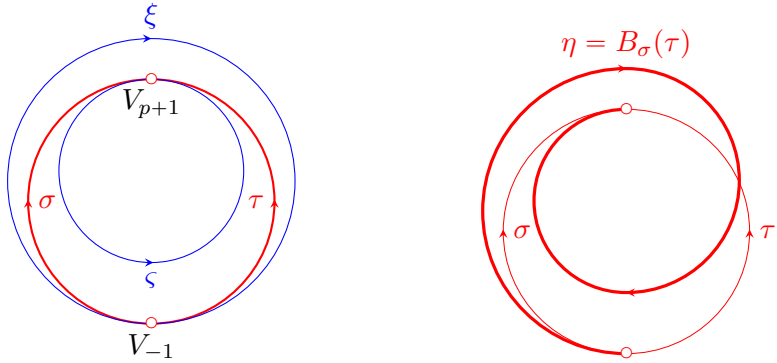
\begin{figure}[ht]\centering
	\begin{tikzpicture}[scale=.27]
	\draw[red,thick](0,0)circle(6);\draw[blue](0,1.5)circle(4.5);
	\draw[blue](0,1)circle(7);
	\draw[red,->,>=stealth](6,0)node[left]{$\tau$}to(6,.1);
	\draw[red,->,>=stealth](-6,0)node[right]{$\sigma$}to(-6,.1);
	\draw[blue,->,>=stealth](0,-3)to(0.1,-3);
	\draw[blue,<-,>=stealth](0,8)to(-.1,8);
	\draw[blue](0,-3)node[below]{$\varsigma$}(0,8)node[above]{$\xi$};
	\draw(0,-6)node[below]{$V_{-1}$}(0,6)node[below]{$V_{p+1}$};
	\draw(0,6)node[white] {$\bullet$} node[red]{$\circ$};\draw(0,-6)node[white]
        {$\bullet$} node[red]{$\circ$};
	\end{tikzpicture}
	\qquad\qquad\qquad
	\begin{tikzpicture}[scale=.27]
	\draw[red,very
    thin](0,0)circle(6);\draw[red](0,8)node[above]{$\eta=B_\sigma(\tau)$};
	\draw[red,->,>=stealth](-6,0)node[right]{$\sigma$}to(-6,.1);
	\draw[red,->,>=stealth](6,0)node[right]{$\tau$}to(6,.1);
	\draw[red,<-,>=stealth](0,-3)to(0.1,-3);
	\draw[red,<-,>=stealth](0,8)to(-.1,8);
	\draw[red,very thick](0,6) arc(90:270:4.5)arc(-90:90:5.5)arc(90:270:7);
	\draw(0,6)node[white] {$\bullet$} node[red]{$\circ$};\draw(0,-6)node[white]
    {$\bullet$} node[red]{$\circ$};
	\end{tikzpicture}
	\caption{The braid twists as compositions of extensions}
	\label{fig}
\end{figure}
	
Now consider the case $\Int_{\bigtriangleup}(\sigma,\tau)=2$, see Figure~\ref{fig}. Let $\varsigma=\tau\wedge\sigma$.
Then we have $$\eta=B_\sigma(\tau)=\overline{\varsigma}\wedge\overline{\sigma}.$$
Similarly, let $\xi=\overline{\tau}\wedge\overline{\sigma}$ and
then we have $$\eta=B_\sigma(\tau)=\overline{\xi}\wedge\sigma.$$
Note that the starting segments of $\overline{\sigma}$, $\overline{\tau}$ and $\overline{\varsigma}$
are in clockwise order at the common starting point, see Figure~\ref{fig}.
By Corollary~\ref{cor:comp}, we have $$\varphi(\overline{\tau},\overline{\varsigma})\circ\varphi(\overline{\sigma},\overline{\tau})=\varphi(\overline{\sigma},\overline{\varsigma}).$$
By Proposition~\ref{prop:key}, the mapping cones of $\varphi(\overline{\tau},\overline{\varsigma})$, $\varphi(\overline{\sigma},\overline{\tau})$ and $\varphi(\overline{\sigma},\overline{\varsigma})$ are in $\widetilde{X}(\sigma)$, $\widetilde{X}(\xi)$ and $\widetilde{X}(\eta)$, respectively. Then applying Octahedral Axiom to this composition gives the following commutative diagram of (images in $\D_{fd}(\Gamma_{\T})/[1]$ of) triangles
\[\xymatrix@C=3pc@R=3pc{
	&\widetilde{X}(\varsigma)\ar[d]\ar@{=}[r]&\widetilde{X}(\varsigma)\ar[d]&\\	
    \widetilde{X}(\xi)\ar[r]\ar@{=}[d]&\widetilde{X}(\eta)\ar[r]\ar[d]&\widetilde{X}(\sigma)\ar[d]^{\varphi(\sigma,\tau)}
		\ar[r]&\widetilde{X}(\xi)\ar@{=}[d]&\\
	\widetilde{X}(\xi)\ar[r]&\widetilde{X}(\sigma)\ar[d]_{\varphi(\overline{\sigma},\overline{\varsigma})}
		\ar[r]^{\varphi(\overline{\sigma},\overline{\tau})}
		&\widetilde{X}(\tau)\ar[d]^{\varphi(\overline{\tau},\overline{\varsigma})}\ar[r]
		&\widetilde{X}(\xi)\\
		&\widetilde{X}(\varsigma)\ar@{=}[r]&\widetilde{X}(\varsigma)&
}\]
Then we have the triangle \eqref{eq:SES}. Since $\varphi(\overline{\tau},\overline{\varsigma})\circ\varphi(\sigma,\tau)=0$ by the triangle in the third column and $\varphi(\overline{\tau},\overline{\varsigma})\circ\varphi(\overline{\sigma},\overline{\tau})=\varphi(\overline{\sigma},\overline{\varsigma})\neq0$,
we deduce that $\varphi(\sigma,\tau)$ and $\varphi(\overline{\sigma},\overline{\tau})$ are linearly independent.
\end{proof}

\subsection{A first formula}
\begin{assumption}\label{ass}
Suppose that $\surf$ admits a triangulation $\T_0$,
such that any two triangles share at most one edge
(i.e. there is no double arrow in the corresponding quiver $Q_{\T_0}$).
\end{assumption}
Let $\T_0=\{\delta_1,\ldots,\delta_n\}$,
and $\T_0^\ast=\{t_1,\ldots,t_n\}$ be the dual of $\T_0$.
Write $\Gamma_0=\Gamma_{\T_0}$ and denote by $\h_0$ the canonical heart of $\D_{fd}(\Gamma_{\T_0})$
with simples $T_i$ corresponding to $t_i$.
By \cite[\S~6]{QQ}, we have the following:
\begin{itemize}
\item The map $\widetilde{X}_{0}=\widetilde{X}_{\T_0}$ in \eqref{eq:X} induces a bijection
$\widetilde{X}_0\colon\CA(\surfo)\xrightarrow{\text{1-1}}\Sph(\Gamma_0)/[1]$
and an isomorphism
\begin{gather}\label{eq:Qmain}
    \iota_0\colon\BT(\T_0)\to\ST(\Gamma_0),
\end{gather}
sending the generator $\Bt{t_i}$ to the generator $\phi_{T_i}$.
\item  There is a commutative diagram
\[\begin{tikzpicture}[xscale=.6,yscale=.6]
\draw(180:3)node(o){$\CA(\surfo)$}(-3,2.2)node(b){\small{$\BT(\surfo)$}}
(0,2.5)node{$\iota_0$}(0,.5)node{$\widetilde{X}_0$};
\draw(0:3)node(a){$\Sph(\Gamma_0)/[1]$}(3,2.2)node(s){\small{$\ST(\Gamma_0)$}};
\draw[->,>=stealth](o)to(a);\draw[->,>=stealth](b)to(s);
\draw[->,>=stealth](-3.2,.6).. controls +(135:2) and +(45:2) ..(-3+.2,.6);
\draw[->,>=stealth](3-.2,.6).. controls +(135:2) and +(45:2) ..(3+.2,.6);
\end{tikzpicture}
\]
in the sense that,
for any $b\in\BT(\surfo)$ and $\eta\in\CA(\surfo)$, we have
\begin{gather}
\label{eq:actions+}
    \widetilde{X}_0\left( b(\eta) \right)
    =\iota_0(b) \left( \widetilde{X}_0(\eta) \right).
\end{gather}
\end{itemize}

We will use $X^0_{\eta}$ to denote an object in the shift orbit $\widetilde{X_0}(\eta)$.

\begin{lemma}\label{lem:00}
Let $\sigma,\tau\in\GCA(\surfo)$ be two oriented arcs sharing the same starting point and let $t_i\in\T_0^\ast$.
Suppose that $t_i,\sigma,\tau$ are pairwise different,
do not intersect each other in $\surf-\Tri$ and
their start segments are in anticlockwise order at the common starting point (cf. Figure~\ref{fig:ab}).
Then $\varphi(\sigma,\tau)\circ\varphi(t_i,\sigma)=0.$
\end{lemma}
\begin{proof}
Suppose that $f=\varphi(\sigma,\tau)\circ\varphi(t_i,\sigma)\neq0$ in $\Hom(T_i,X^0_\tau[m])$
for some integer $m$.
Since $\D_{fd}(\Gamma_{\T_0})$ is 3-Calabi-Yau,
there exists $f^*\in\Hom(X^0_\tau[m],T_i[3])$ such that $f^*\circ f\neq0$.
On the other hand, the start segments of $\sigma$, $t_i$, and $\tau$ are in clockwise order. Then by Corollary~\ref{cor:comp}, $\varphi(\sigma,\tau)=\varphi(t_i,\tau)\circ\varphi(\sigma,t_i)$,
where the morphisms are properly shifted.
Therefore, there is a non-zero composition
\[f^*\circ f\colon
    T_i    \xrightarrow{\varphi(t_i,\sigma)}   X^0_{\sigma}[l]
            \xrightarrow{\varphi(\sigma,t_i)}   T_i[N]
            \xrightarrow{\varphi(t_i,\tau)}   X^0_{\tau}[m]
            \xrightarrow{f^*}   T_i[3].\]
Since $T_i$ is a spherical object, we have $N=0$ or $N=3$.
If $N=0$, then the first two morphisms in the composition above must be identity (up to scale),
which forces $X^0_\sigma[l]=T_i$.
Since $\widetilde{X}_0$ is a bijection, this contradicts $\sigma\neq t_i$.
If $N=3$, similarly we have $X^0_\tau[m]=T_i[3]$, which contradicts $\tau\neq t_i$.
\end{proof}

For any general closed curve $\sigma\in\GCA(\surfo)$, we define $l_0(\sigma)=\Int(\sigma,\T_0)=\sum_{i=1}^n\Int(\sigma,t_i).$

\begin{lemma}\label{lem:decomp}
Let $\sigma$, $\tau$ be two general closed curves in $\GCA(\surfo)$ sharing the same starting point and $\Int_{\surf-\Tri}(\sigma,\tau)=0$. Let $\eta=\tau\wedge\sigma$ and assume that $l_0(\eta)=l_0(\sigma)+l_0(\tau)$. If
\begin{gather}\label{eq:int}
\Int(t_i,\eta)=\Int(t_i,\sigma)+\Int(t_i,\tau),
\end{gather}
holds for some $t_i\in\T_0^\ast$, then we have
\begin{gather}\label{eq:dim}
\dim\Hom^\ZZ(T_i,X^0_\eta)=
\dim\Hom^\ZZ(T_i,X^0_\sigma)
+\dim\Hom^\ZZ(T_i,X^0_\tau).
\end{gather}
\end{lemma}

\begin{proof}
Use the notations for $\sigma$ and $\tau$ in Appendix~\ref{app:1}.
By Proposition~\ref{prop:key}, there is a non-trivial triangle
\[X^0_{\sigma}[-\upsilon]\xrightarrow{\varphi(\sigma,\tau)}
X^0_{\tau}\xrightarrow{\varphi(\overline{\tau},\overline{\eta})}
X^0_{\eta}[l]
\xrightarrow{\varphi(\eta,\overline{\sigma})}
X^0_{\sigma}[-\upsilon+1]\]
for some integers $\upsilon,l$.
Note that, in the case when $\sigma=\tau$, we will apply Proposition~\ref{prop:key}
to $(\eta,\sigma)$ instead of $(\sigma,\tau)$. Nevertheless, we will get the same triangle.
So it is sufficient to prove the map
\begin{gather}\label{eq:mor}\Hom(T_i[r],X^0_\sigma[-\upsilon])
\xrightarrow{\Hom(T_i[r],\varphi(\sigma,\tau))}
\Hom(T_i[r],X^0_\tau)\end{gather}
is zero for any $r\in\ZZ$.
Since $l_0(\eta)=l_0(\sigma)+l_0(\tau)$, by Construction~\ref{cons:above}, we have that $\varphi(\sigma,\tau)$ is of the following form
\[\xymatrix@C=6pc{T_{k_0}\ar[d]^{\varphi_0}\ar@{-}[r]^{\pi_{a_1}\quad}
&T_{k_1}[\varrho_1]\ar@{-}[r]&\cdots\\
T_{j_0}[\upsilon]\ar@{-}[r]^{(-1)^{\upsilon}\pi_{b_1}[\upsilon]\quad}&T_{j_1}[\kappa_1+\upsilon]\ar@{-}[r]&\cdots
}\]
where $\varphi_0$ is induced from the triangle $\Lambda_0$. Then for any $f:T_i[r]\to X^0_\sigma$, the composition $\varphi(\sigma,\tau)\circ f$ is given by $\varphi_0\circ f_0$, where $f_0$ is the component of $f$ from $T_i[r]$ to $T_{k_0}$. Thus, the map \eqref{eq:mor} being zero is equivalent to $\varphi_0\circ f_0=0$.

\parindent=-0pt\textbf{Case I:}
If the common starting point $V_{-1}=W_{-1}$ of $\sigma$ and $\tau$ is not an endpoint of $t_i$,
then $\gamma_i$ is not an edge of the triangle $\Lambda_0$. So $f_0$ is not induced from $\Lambda_0$. Then by Lemma~\ref{lem:comp}, $\varphi_0\circ f_0=0$ as required.

\parindent=-0pt\textbf{Case II:}
If $V_{-1}=W_{-1}$ is an endpoint of $t_i$,
then $f_0$ is the unique nonzero component of $\varphi(t_i,\sigma)$ by Construction~\ref{cons:above}.
Furthermore, \eqref{eq:int} implies that the segment $V_0W_0$ in $\eta$ intersects $t_i$
as shown in Figure~\ref{fig:ab}.
In particular, the start segments of $t_i$, $\sigma$ and $\tau$ are in anti-clockwise order. Hence, by Lemma~\ref{lem:00}, $\varphi(\sigma,\tau)\circ\varphi(t_i,\sigma)=0$, which implies that $\varphi_0\circ f_0=0$ as required.
\begin{figure}[h]\centering
\begin{tikzpicture}[scale=.7]
\draw[blue] (30:3) edge[bend right] (-45:4)(-45:4.5)node[black]{$W_0$}
(-3,0) edge[bend left=20](-50:4)(30:3.5)node[black]{$V_0$}
(-3,0) edge[bend right=15](35:3)
(2.5,-1.5)node{$\eta$}(0,-2)node{$\tau$}(0,1)node{$\sigma$};
\draw[red,thick] (0:-3)node[left]{{$V_{-1}=W_{-1}$}} -- (0:3)node[right]{$t_i$}(1.8,.3)node[black]{{}}
(2.1,0)node[black]{{$\bullet$}};\draw (0:-3)node[white]{$\bullet$}node[red]{$\circ$};
\end{tikzpicture}
\caption{The composition of $\sigma$ and $\tau$}
\label{fig:ab}
\end{figure}
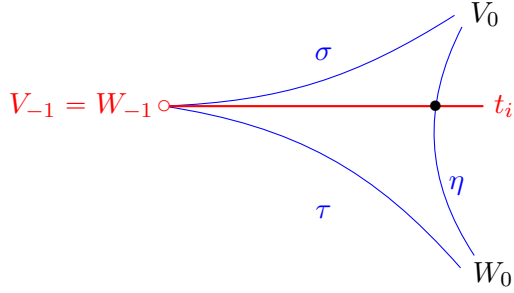

\end{proof}

\begin{proposition}\label{prop:homint}
Let $\eta\in\GCA(\surfo)$.
Under Assumption~\ref{ass}, we have
\begin{gather}\label{eq:homint}
\dim\Hom^\ZZ(T_i,\widetilde{X}_0(\eta))=2\Int(t_i,\eta).
\end{gather}
\end{proposition}
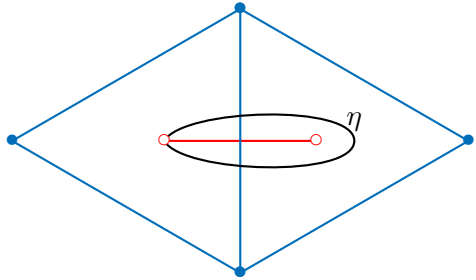
\begin{figure}[hb]\centering
\begin{tikzpicture}[scale=.5]
\draw[red,thick](-2,0)to(2,0);
\draw[NavyBlue,thick](0,3.5)node{$\bullet$}to(0,-3.5)node{$\bullet$}to(-6,0)node{$\bullet$}
to(0,3.5)node{$\bullet$}to(6,0)node{$\bullet$}to(0,-3.5)node{$\bullet$};
\draw[thick](-2,0).. controls +(60:1) and +(90:1) ..(3,0)node[above]{$\eta$}
.. controls +(-90:1) and +(-60:1) ..(-2,0);
\draw(-2,0)node[white]{$\bullet$}node[red]{$\circ$}(2,0)node[white]{$\bullet$}node[red]{$\circ$};
\end{tikzpicture}
\caption{A loop encloses a closed arc}\label{fig:L}
\end{figure}

\begin{proof}
When $\Int(t_i,\eta)<1$ (i.e. is $0$ or $1/2$),
the formula is proved in \cite[Proposition~5.9]{QQ}.
Now we assume that $\Int(t_i,\eta)\geq 1$.
Use induction on $l_0(\eta)$ starting with the trivial cases $l_0(\eta)=1$ (i.e. $\eta=t_j$ for some $j$).
Now suppose that the formula holds for $l_0(\eta)\leq l$, with some $l\geq 1$.
Consider the case $l_0(\eta)=l+1$. Then there are three cases:

\parindent=-0pt\textbf{(I):} $\eta$ intersects a triangle (with decorating point $Z$)
that does not intersect $t_i$.
Applying \cite[Lemma~3.14]{QQ} to decompose $\eta$ into $\sigma$ and $\tau$,
w.r.t. $Z$. Since $Z$ is not an endpoint of $t_i$, \eqref{eq:int} holds. So by Lemma~\ref{lem:decomp}, we have \eqref{eq:dim} for $\sigma$ and $\tau$
(by properly choosing their orientations; same for later use of this lemma). By the induction hypothesis, \eqref{eq:homint} holds for $\sigma,\tau$
and hence holds for $\eta$ too by \eqref{eq:dim}.

\parindent=-0pt\textbf{(II):} $\Int_{\surf-\Tri}(\eta,t_i)\neq0$.
  Let $Z$ be an endpoint of $t_i$, such that the triangle containing $Z$ contains intersections of $\eta$ and $t_i$ in $\surf-\Tri$.
  Choose the closest intersection $Y$ between $\eta$ and $t_i$ from $Z$.
  Applying \cite[Lemma~3.14]{QQ}, w.r.t. $Z$ and the line segment $YZ(\subset t_i)$
  to decompose $\eta$ into $\sigma$ and $\tau$.
  Again they satisfy the condition \eqref{eq:int} in Lemma~\ref{lem:decomp}
  and thus \eqref{eq:dim} holds.
  By the induction hypothesis, \eqref{eq:homint} holds for $\sigma,\tau$ and hence for $\eta$.

\parindent=-0pt\textbf{(III):} Suppose the conditions in \text{(I)} and \text{(II)} both fail, i.e.
\begin{itemize}\item $\Int_{\surf-\Tri}(\eta,t_i)=0$ and
\item $\eta$ is contained in the two triangles of $\T_0$ which intersects $t_i$.
\end{itemize}
Note that these two triangles share exactly one edge by Assumption~\ref{ass}.
Then since $\eta\neq t_i$, we deduce that $\eta$ is a loop encloseing $t_i$ (cf. Figure~\ref{fig:L}).
In this case, $X^0_\eta=\Cone(T_i\to T_i[3])[-1]$.
A direct calculation shows that \eqref{eq:homint} holds.
\end{proof}

\begin{corollary}\label{cor:int}
Under Assumption~\ref{ass},
\begin{gather}\label{eq:homint2x}
\dim\Hom^\ZZ(\widetilde{X}_0(\eta_1),\widetilde{X}_0(\eta_2))=2\Int(\eta_1,\eta_2),
\end{gather}
for any $\eta_i\in\CA(\surfo)$.
\end{corollary}
\begin{proof}
By Proposition~\ref{prop:r}, there exists $t_i\in\T_0^*$ and $b\in\BT(\T_0)$ such that $\eta_1=b(t_i)$.
Then by \eqref{eq:actions+}, we have $\widetilde{X}_0(b(t_i))=\iota(b)\left(\widetilde{X}_0(t_i)\right)$.
Hence
\[\begin{array}{rcl}
\dim\Hom^\ZZ(\widetilde{X}_0(\eta_1),\widetilde{X}_0(\eta_2))&=&
\dim\Hom^\ZZ(\widetilde{X}_0(t_i),\widetilde{X}_0(b^{-1}(\eta_2)))\\
&=&2\Int(t_i,b^{-1}(\eta_2))\\
&=&2\Int(\eta_1,\eta_2).
\end{array}\]
\end{proof}

\section{Main results}\label{Sec:4}
\subsection{Independence}\label{sec:ind}
We still assume Assumption~\ref{ass} holds in this subsection, i.e. there is an initial triangulation $\T_0$ of $\surfo$ such that any two triangles share at most one edge.
Recall from \cite{QQ}, that two elements $\psi$ and $\psi'$ in $\Aut\D_{fd}(\Gamma_0)$ are \emph{isotopic},
denote by $\psi\sim\psi'$,
if $\psi^{-1}\circ\psi'$ acts trivially on $\Sph(\Gamma_0)$.
Let $$\Aut^\circ\D_{fd}(\Gamma_0)=\Aut\D_{fd}(\Gamma_0)/\sim.$$
By \cite[(2.6)]{QQ}, $\psi\sim\psi'$ is equivalent to the condition that
$\psi^{-1}\circ\psi'$ acts trivially on $\Sim\h_0$.
We will say an element $\varphi$ in $\Aut\D_{fd}(\Gamma_0)$ is null-homotopic
if $\varphi\sim\id$.

Now let $\T$ be an arbitrary triangulation. Keep the notations in Section~\ref{sec:twist}.
Denote by $\h_{\T}$ the canonical heart in $\D_{fd}(\Gamma_\T)$ with simples $\{S_i\}$ corresponding to open arcs in $\T^\ast=\{s_i\}$.
Denote by $\Sph(\Gamma_\T)$ the set of reachable spherical objects.
\begin{definition}
We say two derived equivalences $\phi,\phi'\colon\D_{fd}(\Gamma_0)\to\D_{fd}(\Gamma_\T)$
are isotopic if they only differ by null-homotopies, i.e.
$\phi'=\varphi_1\circ\phi\circ\varphi_0$ for some $\varphi_0\in\Aut\D_{fd}(\Gamma_0)$
and $\varphi_1\in\Aut\D_{fd}(\Gamma_\T)$, which are null-homotopic.
\end{definition}

\begin{proposition}\label{lem:u}
There is a unique exact equivalence $\Phi_\T\colon\D_{fd}(\Gamma_0)\to\D_{fd}(\Gamma_\T)$,
up to isotopy and shifts,
such that it induces a bijection
$$\Phi_{\T}\colon\Sph(\Gamma_0)/[1] \to\Sph(\Gamma_\T)/[1]$$
satisfying the following condition:
\[\bullet\text{
for any $s\in\T^*$, the corresponding simple in $\Sim\h_{\T}$
is in the shift orbit $\Phi_{\T}(\widetilde{X}_0(s))$}.
\]
\end{proposition}
\begin{proof}
First we show the uniqueness.
Suppose that there are two such derived equivalences $\Phi_{\T}$ and $\Phi'_{\T}$.
Then we have $\Phi_{\T}^{-1}\circ\Phi'_{\T}(T_i)=T_i[m_i]$ for any simple $T_i$ in
the canonical heart $\h_0$.
By calculating the $\Hom^\ZZ$, we deduce that all $m_i$ should coincide, i.e.
$\Phi_{\T}^{-1}\circ\Phi'_{\T}\circ[-m]$ preserves $\Sim\h_0$ and hence $\Sph(\Gamma_0)$.
In other words, $\Phi_{\T}^{-1}\circ\Phi'_{\T}\circ[-m]$ is null-homotopic in $\Aut\D_{fd}(\Gamma_0)$,
as required.

Now we prove the existence by induction, on the minimal number of flips from $\T_0$ to $\T$,
starting from the trivial case.
Now suppose that $\T$ admits a required derived equivalence $\Phi_{\T}$, i.e.
\begin{gather}\label{eq:T}
    \Phi_{\T}(\widetilde{X}_0({s_i}))=S_i[\ZZ].
\end{gather}
Then we only need to show that there exists a required derived equivalence $\Phi_{\T'}$ for any flip $\T'$ of $\T$
in $\surfo$.

Without loss of generality, suppose that $\T'$ is the forward flip of $\T$ w.r.t. an arc $\gamma_j$
and let $s_j$ be the dual arc of $\gamma_j$ in $\T^*$.
By \cite{KY},
there is an exact equivalence $\Phi\colon\D(\Gamma_{\T})\to\D(\Gamma_{\T'})$ satisfying
$$\Phi\left( \tilt{(\h_\T)}{\sharp}{S_j} \right)=\h_{\T'},$$
where $\tilt{(\h_\T)}{\sharp}{S_j}$ is the simple forward tilt of $\h_\T$,
w.r.t. $S_j$ (cf. \cite[\S~5]{KQ}).

Let $(\T')^*$ consist of closed arcs $s_i'$ and $\Sim\h_{\T'}$ consist of the corresponding simples $S_i'$. By the tilting formula in \cite[Proposition~5.2]{KQ}, we have
  \[
    \Phi^{-1}(S_i')=\begin{cases}
      \phi^{-1}_{S_j}(S_i), &\text{if there are arrows from $i$ to $j$ in $Q_{\T}$},\\
      S_{j}[1],&\text{if $i=j$,}\\
      S_{i}, &\text{otherwise}.\\
    \end{cases}
  \]
On the other hand, note that the indexing of $s_i'$ is induced by the indexing of $s_i$ via the Whitehead move (cf. \cite[Figure~10]{QQ}). It is straightforward to see that
$$
    \quad\quad\;\; s_i'=\begin{cases}
    B_{s_j}(s_i),& \text{if there are arrows from $i$ to $j$ in $Q_{\T}$},\\
    s_i,& \text{otherwise}.
    \end{cases}
$$
Taking $b=\bt{s_j}$ with $\iota_0(b)=\phi^{-1}_{\widetilde{X}_0(s_j)}$ and $\eta=s_i$,
\eqref{eq:actions+} becomes
\begin{gather}\label{eq:35}
    \phi^{-1}_{ \widetilde{X}_0({s_j}) }  (\widetilde{X}_0({s_i})  )
    = \widetilde{X}_0(  \bt{s_j}(s_i) ).
\end{gather}
Then by \eqref{eq:T}, we have
\begin{gather}\begin{array}{rl}
    &\Phi^{-1}(S_i'[\ZZ])=\phi^{-1}_{S_j}(S_i)[\ZZ]\\
    =&\phi^{-1}_{ \Phi_{\T}(\widetilde{X}_0({s_j})) }  \left(\Phi_{\T}(\widetilde{X}_0({s_i}))\right)\\
    =&\Phi_{\T}  \left(  \phi^{-1}_{ \widetilde{X}_0({s_j}) }  (\widetilde{X}_0({s_i})  ) \right)\\
    =&\Phi_{\T}  \left(   \widetilde{X}_0(  \Bt{s_j}(s_i)  ) \right)\\
    =&\Phi_{\T} \left(   \widetilde{X}_0(  s_i'  ) \right) .
\end{array}\end{gather}
if there are arrows from $i$ to $j$ in $Q_{\T}$.
Note that for other $i$ the equation above also holds automatically.
Thus $S_i'[\ZZ]=\Phi\circ\Phi_{\T}(\widetilde{X}_0({s_i'}))$
and $\Phi_{\T'}=\Phi_{\T'}\circ\Phi_{\T}$ is the required equivalence.
\end{proof}

Recall that there is a bijection
$\widetilde{X}_0\colon \CA(\surfo)\rightarrow \Sph(\Gamma_{0})/[1]$
and we proceed to discuss $\widetilde{X}_\T.$

\begin{proposition}\label{thm:ind}
$\widetilde{X}_\T$ induces a bijection
$\widetilde{X}_\T\colon\CA(\surfo)\rightarrow \Sph(\Gamma_{\T})/[1]$,
that fits into the following commutative diagram
\begin{gather}\label{eq:diaa}\xymatrix{
    & \CA(\surfo) \ar[dl]_{\widetilde{X}_0} \ar[dr]^{\widetilde{X}_{\T}}\\
    \Sph(\Gamma_0)/[1] \ar[rr]^{\Phi_{\T}} && \Sph(\Gamma_\T)/[1],
}\end{gather}
where $\Phi_{\T}$ is the bijection in Proposition~\ref{lem:u}.
\end{proposition}
\begin{proof}
Since $\widetilde{X}_0$ and $\Phi_{\T}$ are bijections, we only need to prove
\begin{gather}\label{eq:cc}
    \widetilde{X}_{\T}(\eta)=\Phi_\T\circ\widetilde{X}_0(\eta)
\end{gather}
Use induction on $l_0(\eta)=\Int(\eta,\T)$.
The starting step $(l_0(\eta)=1)$ is covered by Proposition~\ref{lem:u}.
Now let us deal the inductive step for some $\eta$ with $l_0(\eta)>1$ while assuming that \eqref{eq:cc} holds for any
$\eta'$ with $l_0(\eta')< l_0(\eta)$.
By Lemma~\ref{lem:dcp}, there are $\alpha$ and $\beta$ with the corresponding conditions there.
Without loss of generality, assume that $\eta=B_\alpha(\beta)$. By inductive assumption, we have
\begin{gather}\label{eq:and}
        \widetilde{X}_{\T}(\alpha)=\Phi_\T\circ\widetilde{X}_0(\alpha) \quad\text{and}\quad
        \widetilde{X}_{\T}(\beta)=\Phi_\T\circ\widetilde{X}_0(\beta).
\end{gather}
Since by Corollary~\ref{cor:int}, we have
\begin{gather}\label{eq:notice}
\dim\Hom^\ZZ(\widetilde{X}_{0}(\alpha),\widetilde{X}_{0}(\beta))=2\Int(\alpha,\beta),
\end{gather}
then the triangles in Proposition~\ref{pp} implies that
\begin{gather}\label{eq:notice2}
\widetilde{X}_{0}(\eta)=
    \phi_{ \widetilde{X}_{0}(\alpha) }
    \left( \widetilde{X}_{0}(\beta) \right).
\end{gather}
Notice that by \eqref{eq:and}, the equalities \eqref{eq:notice}
and \eqref{eq:notice2} also hold for $\widetilde{X}_\T$.
Hence
\[\begin{array}{rll}
    \widetilde{X}_{\T}(\eta)&=&
    \phi_{ \widetilde{X}_{\T}(\alpha) } \left( \widetilde{X}_{\T}(\beta) \right)\\
    &=&\phi_{ \Phi_{\T}\left(\widetilde{X}_0(\alpha)\right) }\left( \Phi_{\T}\circ\widetilde{X}_0(\beta) \right)\\
    &=&\Phi_{\T}\left(  \phi_{ \widetilde{X}_0(\alpha) }\left( \widetilde{X}_0(\beta) \right)  \right)\\
    &=&\Phi_{\T}\circ\widetilde{X}_0(\eta)
\end{array}\]
as required.
\end{proof}

\begin{remark}\label{rem}
By the proposition above,
one can identify all sets $\Sph(\Gamma_\T)$ of reachable spherical objects,
for any $\T$ in $\EGp(\surfo)$,
using the canonical derived equivalences in Proposition~\ref{lem:u} between $\D_{fd}(\Gamma_\T)$.
Hence, such equivalences also allow us to identify all the spherical twist groups $\ST(\Gamma_\T)$.
Note that here we will consider $\ST(\Gamma_\T)$ as a subgroup of $\Aut^\circ\D_{fd}(\Gamma_\T)$.
\end{remark}

\subsection{The first formula revisit}
\begin{theorem}\cite[Conjecture~10.5]{QQ}\label{thm:1}
For any triangulation $\T$ and  $\eta_i\in\CA(\surfo)$, we have
\begin{gather}\label{eq:homint2}
\dim\Hom^\ZZ(\widetilde{X}_\T(\eta_1),\widetilde{X}_\T(\eta_2))=2\Int(\eta_1,\eta_2).
\end{gather}
\end{theorem}
\begin{proof}
If Assumption~\ref{ass} holds, the theorem is equivalent to Corollary~\ref{cor:int}
since one can identify all bijections $\widetilde{X}_\T$ as in Remark~\ref{rem}.
For the special cases that $\surf$ does not satisfy Assumption~\ref{ass}, that is, \begin{itemize}
\item either $\surf$ is an annulus with one marked point in each boundary component,
\item or $\surf$ is a torus with one boundary component and one marked point,
\end{itemize}
one can apply the same method in \cite[\S~7]{QQ}.
More precisely, the formula holds for a higher rank surface
(e.g. the surface obtained from $\surf$ by adding a marked point)
and hence also holds for $\surf$.
\end{proof}


\subsection{Independence revisit}
We use intersection formula \eqref{eq:homint2}
to prove \eqref{eq:35} for $\T$.
\begin{proposition}\label{pp:actions+}
For any $\sigma,\tau\in\CA(\surfo)$ with $\Int_{\surf-\Tri}(\sigma,\tau)=0$, we have
\begin{gather}
\label{eq:actions-}
    \widetilde{X}_\T\left( B_{\tau}^{\varepsilon}(\sigma) \right)
    =\phi^{\varepsilon}_{ \widetilde{X}_\T(\tau) }  \left( \widetilde{X}_\T(\sigma) \right),
    \quad \varepsilon\in\{\pm1\}.
\end{gather}
\end{proposition}
\begin{proof}
Without lose of generality, we only prove the formula for $\varepsilon=1$.
On one hand, by \eqref{eq:homint2}, we have $$\dim\Hom^\ZZ(\widetilde{X}_\T(\eta),\widetilde{X}_\T(\tau))=\Int_{\Tri}(\sigma,\tau).$$
On the other hand, there is a triangle, i.e. \eqref{eq:SES} or \eqref{eq:SES1}, in Proposition~\ref{pp}
with $\eta=B_\sigma(\tau)$.
Then
$\widetilde{X}_\T(\eta)=\phi_{\widetilde{X}_\T(\sigma)}(\widetilde{X}_\T(\tau))$ as required.
\end{proof}

\begin{proposition}\label{thm:INDE}
For any $\surf$ and initial triangulation $\T_0$ (without Assumption~\ref{ass}),
Proposition~\ref{lem:u} and Proposition~\ref{thm:ind} holds.
\end{proposition}
\begin{proof}
Basically, we follow the same proof there.
Note that \eqref{eq:35} in the proof of Proposition~\ref{lem:u} is now covered by the proposition above,
which enables us do the generalization.
\end{proof}

\subsection{Intersection between open and closed arcs}
Let $\Gamma$  be the Ginzburg dg algebra of some quiver with potential.
A \emph{silting set} $\P$ in a triangulated category $\D$ is an $\Ext^{>0}$-configuration, i.e. a maximal collection
of non-isomorphic indecomposables such that $\Ext^i(P,T)=0$ for any $P,T\in\P$ and integer $i>0$.
The silting object associated to $\P$ is $\bigoplus_{T\in\P}T$.
By abuse of notation, we will not distinguish a silting set and its associated silting object.
For example, $\Gamma$ is the canonical silting object/set in $\per\Gamma$.

Moreover, one can forward/backward mutate a silting object to get new ones (cf. \cite{AI} for details).
A silting set $\P$ in $\per\Gamma$ is \emph{reachable} if it can be obtained by repeatedly mutations from $\Gamma$.
Denote by $\SEGp(\Gamma)$ the set of reachable silting sets in $\per\Gamma$
and
$$\RR(\per\Gamma)=\bigcup_{\P\in\SEGp(\Gamma)} \P$$
the set of \emph{reachable rigid objects} in $\per\Gamma$.
Recall a result from \cite{QQ2}.

\begin{lemma}\cite[Theorem~3.6]{QQ2}
There is a canonical bijection
\[\widetilde{M}_\T\colon\OAp(\surfo)\to\RR(\per\Gamma_\T)\]
where $\OAp(\surfo)$ is the subset of $\OA(\surfo)$ consisting of the open arcs in some triangulation in $\EGp(\surfo)$.
\end{lemma}

We finish the paper by proving another conjecture in \cite{QQ}.

\begin{theorem}\cite[Conjecture~10.6]{QQ}\label{thm:2}
For any triangulation $\T$, $\gamma\in\OAp(\surfo)$ and $\eta\in\CA(\surfo)$, we have
\begin{gather}\label{eq:f2}
    \dim\Hom^\ZZ(\widetilde{M}_\T(\gamma),\widetilde{X}_\T(\eta))=\Int(\gamma,\eta).
\end{gather}
\end{theorem}
\begin{proof}
First, for any two triangulations $\T$ and $\T'$,
we actually have a canonical identification $\Phi\colon\D_{fd}(\Gamma_\T)\to\D_{fd}(\Gamma_{\T'})$,
as shown in Proposition~\ref{thm:INDE}.
Note that there is a simple-projective duality between
a silting set in $\per\Gamma$ and the set of simples of the corresponding heart in $\D_{fd}(\Gamma)$.
Thus, as $\Phi$ preserves reachable spherical objects, up to shift,
$\Phi$ preserves reachable rigid objects, up to shift.
Second, by \cite[Lemma~5.13]{QQ}, the theorem holds for $\gamma\in\T$ and any $\eta\in\OA(\surf)$.
Now, choose any $\gamma\in\OAp(\surfo)$.
Let $\T'$ be a triangulation in $\EGp(\surfo)$ that contains $\gamma$.
Then we have
\[
    \dim\Hom_{\D_{fd}(\Gamma_\T)}^\ZZ(\widetilde{M}_\T(\gamma),\widetilde{X}_\T(\eta))=
    \dim\Hom_{\D_{fd}(\Gamma_{\T'})}^\ZZ(\widetilde{M}_{\T'}(\gamma),\widetilde{X}_{\T'}(\eta))=
    \Int(\gamma,\eta).
\]
\end{proof}

\appendix
\section{The string model}\label{app:A}
\subsection{Homological preparation}\label{sec:HH}

Let $(Q,W)=(Q_\T,W_\T)$ be the quiver with potential associated to a triangulation $\T$ of an unpunctured marked surfaces $\surf$. Recall from Section~\ref{subsec:dgquiver} that there is an associated graded quiver $\overline{Q}$ and an associated Ginzburg dg algebra $\Gamma_\T$ whose underlying graded algebra is the completion of the graded path algebra $\k\overline{Q}$.

For each vertex $i$ of $\overline{Q}$, denote by $S_i$ the corresponding simple module of $\Gamma_\T$.
There is a canonical heart $\h_\T$ in $\D_{fd}(\Gamma_\T)$, whose simples are $S_1,\cdots,S_n$. Let $$S_\T=\bigoplus_{i=1}^n S_i$$
the direct sum of the simples in $\h_\T$.
Consider the differential graded endomorphism algebra
$\ee_\T=\RHom(S_\T, S_\T).$
By \cite{KN}, we have the following exact equivalence:
\begin{gather}\label{eq:DE}
\xymatrix@C=4pc{
    \D_{fd}(\Gamma_\T) \ar@<.5ex>[rr]^{ \RHom_{\Gamma_\T}(S_\T, ?) }
        \ar@{<-}@<-.5ex>[rr]_{ ?\otimes^{\mathbf{L}}_{\ee_{\T}}S_\T }  && \per\ee_\T,
}\end{gather}
In particular, the simples in $\h_\T$ become the
indecomposable direct summands of $\ee_{\T}$.
Then the Ext-algebra $\EE_\T=\oplus_{i\in\mathbb{Z}}\Ext_{\D_{fd}(\Gamma_\T)}^i(S_\T,S_\T)$ is isomorphic to the homology algebra of $\ee_\T$. A basis of $\EE_\T$ is indexed by the arrows and trivial paths in $\overline{Q}$ as follows.

\begin{lemma}[Lemma~2.15 in \cite{KY}]\label{lem:simhom}
Let $i,j$ be vertices of $\overline{Q}$. Then $\Hom_{\D_{fd}(\Gamma)}(S_i,S_j[r])$ has a basis
\begin{gather}\label{eq:pi_0}
    \{\pi_{b}\mid\text{$b:i\to j\in \overline{Q}_1$ with $\deg b=1-r$}\}\cup
    \{\pi_{e_i}=\id_{S_i}\mid \text{if $i=j$ and $r=0$}\}
\end{gather}
where $e_i$ is the trivial path at $i$.
\end{lemma}

There is an $A_\infty$ structure on $\EE_\T$, induced
by the differential of $\Gamma_\T$ (cf. \cite[Appendix~A.15]{Ke}).
Using this structure, we have the following result.

\begin{lemma}\label{lem:comp}
The dg algebra $\ee_\T$ is formal and hence is quasi-isomorphic to $\EE_\T$. Moreover, for any trivial paths $e_i$ and $e_j$ and any arrows $x$ and $y$ in $\overline{Q}$, we have the following.
\begin{itemize}
	\item[(1)] \begin{gather}\pi_{e_j}\circ\pi_{e_i}=
	\begin{cases}\pi_{e_i}& \text{if $i=j$;}\\ 0&\text{otherwise.}\end{cases}\end{gather}
	\item[(2)] \begin{gather}\pi_{y}\circ\pi_{e_i}=
	\begin{cases}\pi_{y}& \text{if $s(y)=i$;}\\ 0&\text{otherwise.}\end{cases}\end{gather}
	\begin{gather}\pi_{e_j}\circ\pi_{x}=
	\begin{cases}\pi_{x}& \text{if $t(x)=j$;}\\ 0&\text{otherwise.}\end{cases}\end{gather}
	\item[(3)] \begin{gather}\pi_y\circ\pi_x=
	\begin{cases}\pi_{\alpha^\ast}& \text{if $xy\alpha$ (up to cyclical equivalence) is a term in $W_\T$;}\\ \pi_{t_{s(x)}} & \text{if $y=x^\ast$ or $x=y^\ast$} \\ 0&\text{otherwise.}\end{cases}\end{gather}
\end{itemize}
Here, $s(\alpha)$ denotes the starting point of an arrow $\alpha$ and $t(\alpha)$ denotes the ending point of $\alpha$.
\end{lemma}


By Lemma~\ref{lem:comp}, there is an exact equivalence $\per\ee_\T\simeq\per\EE_\T$, which, together with the equivalence \eqref{eq:DE}, gives an exact equivalence $\D_{fd}(\Gamma_\T)\simeq\per\EE_\T$. We will identify $\D_{fd}(\Gamma_\T)$ with $\per\EE_\T$ when there is no confusion. In particular, $S_1,\cdots,S_n$ become the indecomposable direct summands of $\EE_\T$ as dg $\EE_\T$-modules. Since the differential of $\EE_\T$ is zero, morphisms in \eqref{eq:pi_0} become homomorphisms of dg $\EE_\T$-modules, and in particular maps.

\begin{description}
\item[Convention] Let $a\colon i\to j\in\overline{Q}_1$.
By abuse of notation, $\pi_a[m]$ in $$\Hom_{\D_{fd}(\Gamma_\T)}(S_i[m],S_j[m+1-\deg a])$$
will
all be denoted by $\pi_a$ for short.
In particular, $\pi_b\pi_a$ makes sense whenever $ba\neq0$.
Similarly for other morphisms.
\end{description}

\subsection{The string model}\label{app:1}

To each internal point $A$ in an arc $\gamma_i\in\T$, we associate a vertex $\nu_A:=i$. Let $l$ be a segment in a triangle $\Delta$ of $\T$, whose endpoints are internal points in sides of $\Delta$ and which is not isotopic to a segment of any side of $\Delta$. Let $A,B$ be the endpoints of $l$ such that from $A$ to $B$ the decorating point in $\Delta$ is to the right of $l$. We associate a graded arrow $\alpha(l)$ to be the unique arrow from $\nu_A$ to $\nu_B$ in $\overline{Q}$ induced from $\Delta$. See Figure~\ref{fig:mors}.


\begin{figure}[ht]\centering
\begin{tikzpicture}[yscale=.6,xscale=.8,rotate=0]
\draw[NavyBlue,thick](-2,0)node{$\bullet$}to(2,0)node{$\bullet$}to(0,4)node{$\bullet$}to(-2,0)
    (0,1.5)node[red]{$\circ$};
\draw[NavyBlue,thick](0,0)node[below]{$\gamma_k$} (-1.4,1.5)node[above]{$\gamma_i$} (1.4,1.5)node[above]{$\gamma_j$};
\end{tikzpicture}
\qquad\qquad
\begin{tikzpicture}[yscale=.45,xscale=.6,rotate=0]
\draw[NavyBlue,thick](-2,4)node{$\bullet$} (2,4)node{$\bullet$} (0,0)node{$\bullet$};
\draw[NavyBlue,thick](-2,4)[left]node{$i$} (2,4)node[right]{$j$} (0,0)node[below]{$k$};

\draw[thick,->](-1.6,4.2) to (1.6,4.2);
\draw[thick,->](1.6,3.9) to (-1.6,3.9);
\draw[thick](0,4.2)node[above]{$a$} (0,3.9)node[below]{$a^\ast$};

\draw[thick,->](-.4,.4) to (-2.2,3.6);
\draw[thick,->](-1.9,3.6) to (-.1,.4);
\draw[thick](-1.3,2)node[left]{$c$} (-1,2.15)node[right]{$c^\ast$};

\draw[thick,->](2.2,3.6) to (.4,.4);
\draw[thick,<-](1.9,3.6) to (.1,.4);
\draw[thick](1.3,2)node[right]{$b$}  (1.1,2.1)node[left]{$b^\ast$};

\draw[thick,->](-2,4.3) .. controls +(70:2) and +(150:2).. (-2.5,4);
\draw[thick](-2.5,5.2)node[above]{$t_i$};

\draw[thick,<-](2,4.3) .. controls +(110:2) and +(30:2).. (2.5,4);
\draw[thick](2.5,5.2)node[above]{$t_j$};

\draw[thick,<-](.2,-.2)..controls +(-40:2.5) and +(-140:2.5)..(-.2,-.2);
\draw[thick](0,-1.8)node{$t_k$};
\end{tikzpicture}

\qquad

\begin{tikzpicture}[yscale=.6,xscale=.8,rotate=0]
\draw[NavyBlue,thick](-2,0)node{$\bullet$}to(2,0)node{$\bullet$}to(0,4)node{$\bullet$}to(-2,0)
    (0,1.5)node[red]{$\circ$};
\draw[NavyBlue,thick](0,0)node[below]{$\gamma_k$} (-1.4,1.5)node[above]{$\gamma_i$} (1.4,1.5)node[above]{$\gamma_j$};
\draw[red](-.8,2.7)node[above]{$A$} (.9,2.5)node[above]{$B$};
\draw[red, thick,>=stealth](-.5,3)edge[bend right=30](.5,3);
\draw[red](0,2.5)node{$l$};
\draw[, thick] (0,-2)node{$\alpha(l)=a:\nu_A\to \nu_B$};
\end{tikzpicture}
\qquad
\begin{tikzpicture}[yscale=.6,xscale=.8,rotate=0]
\draw[NavyBlue,thick](-2,0)node{$\bullet$}to(2,0)node{$\bullet$}to(0,4)node{$\bullet$}to(-2,0)
    (0,1.5)node[red]{$\circ$};
\draw[NavyBlue,thick](0,0)node[below]{$\gamma_k$} (-1.4,1.5)node[above]{$\gamma_i$} (1.4,1.5)node[above]{$\gamma_j$};
\draw[red](-.8,2.7)node[above]{$A$} (1,0)node[below]{$B$};
\draw[red, thick](-.5,3) .. controls +(30:1) and +(30:1) .. (1,0);
\draw[red](.5,1.8)node{$l$};
\draw[thick] (0,-2)node{$\alpha(l)=c^\ast:\nu_A\to \nu_B$};
\end{tikzpicture}
\qquad
\begin{tikzpicture}[yscale=.6,xscale=.8,rotate=0]
\draw[NavyBlue,thick](-2,0)node{$\bullet$}to(2,0)node{$\bullet$}to(0,4)node{$\bullet$}to(-2,0)
    (0,1.5)node[red]{$\circ$};
\draw[NavyBlue,thick](0,0)node[below]{$\gamma_k$} (-1.2,1.9)node[above]{$\gamma_i$} (1.2,1.9)node[above]{$\gamma_j$};
\draw[red, thick] (-.5,3).. controls +(-60:.1) and +(60:.5)..(.5,1.25);
\draw[red, thick] (-1.5,1).. controls +(0:.1) and +(-120:.5)..(.5,1.25);
\draw[red](.7,1)node{$l$};
\draw[, thick] (0,-2)node{$\alpha(l)=t_i:\nu_A\to \nu_B$};
\draw[red](-.8,2.7)node[above]{$A$} (-1.8,1.5)node[below]{$B$};
\end{tikzpicture}

\caption{Segments inducing graded arrows}
\label{fig:mors}
\end{figure}
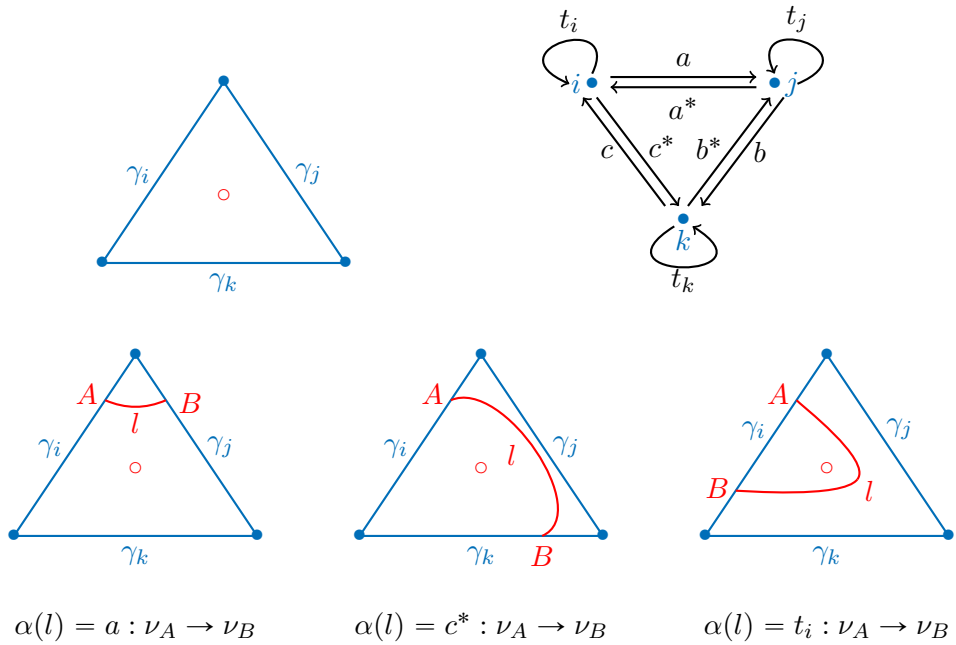

\begin{figure}\centering
\begin{tikzpicture}[yscale=.6,xscale=.8,rotate=0]
\draw[NavyBlue, thick](-2,0)node{$\bullet$}to(2,0)node{$\bullet$}to(0,4)node{$\bullet$}to(-2,0)
    (0,1.5)node[red]{$\circ$};
\draw[red, thick,>=stealth](-1.6,-1) .. controls +(60:.1) and +(-120:.2) .. (-1.1,0);
\draw[red, thick,>=stealth](-1.1,0) .. controls +(60:.1) and +(180:.5) .. (0,.9);

\draw[red, thick,>=stealth](1.1,0) .. controls +(120:.1) and +(0:.5) .. (-.1,.9);
\draw[red, thick,>=stealth](1.6,-1) .. controls +(120:.1) and +(-60:.2) .. (1.1,0);

\draw(-1.1,0)node{$\bullet$}node[above]{}(1.1,0)node{$\bullet$}node[above]{};
\end{tikzpicture}
\caption{A digon intersected by $\sigma$ and $\T$}
\label{fig:bad}
\end{figure}
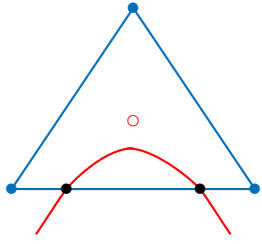

\begin{construction}\label{cons:string}
Let $\sigma$ be an oriented general closed arc such that
it is in a minimal position w.r.t. $\T$ (i.e. there is no digon shown as in Figure~\ref{fig:bad}).

\begin{itemize}
\item Suppose that $\sigma$ intersects $\T$ at
  $V_0,\ldots,V_p$ accordingly from its starting point to its ending point, where $V_i$ is in the arc $\gamma_{k_i}\in\T$
  for $0\leq i\leq p$ and some $1\leq k_i\leq n$ (cf. Figure~\ref{fig:int0}).
  Moreover, denote by $V_{-1}$ and $V_{p+1}$ the starting and ending points of $\sigma$, respectively, by $\sigma(a,b)$ the segment of $\sigma$ between $V_a$ and $V_b$ for $0\leq a<b\leq p+1$, and by $\Lambda_i$ the triangle containing the segment $\sigma(i-1,i)$ for $0\leq i\leq p+1$.
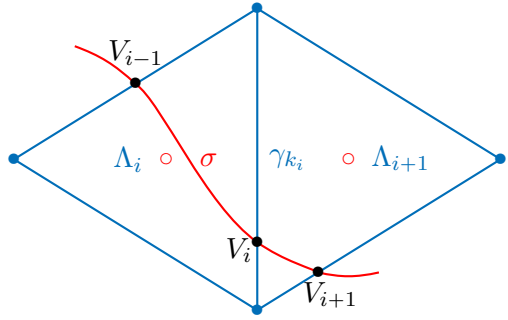
\begin{figure}[ht]\centering
\begin{tikzpicture}[yscale=1,xscale=.8,rotate=90]
\draw[NavyBlue, thick](-2,0)node{$\bullet$}to(2,0)node{$\bullet$}to(0,4)node{$\bullet$}to(-2,0)
    to(0,-4)node{$\bullet$}to(2,0)
    (0,1.5)node[red]{$\circ$} (0,1.7)node[left]{{$\Lambda_i$}}
    (0,-1.5)node[red]{$\circ$}(0,-1.7)node[right]{$\Lambda_{i+1}$};
\draw[](-1.8,-1.2)node[]{$V_{i+1}$} (1.4,2)node[]{$V_{i-1}$}
    (-1.2,.3)node[]{$V_{i}$};
\draw[red, thick,>=stealth](-1.5, -2)to[bend left=10](-1.5,-1)
    .. controls +(60:.1) and +(-120:.5) .. (-1.1,0);
\draw[red, thick,>=stealth](-1.1,0) .. controls +(60:1) and +(-120:.5) .. (1,2)
    to[bend left=-10](1.5,3);
\draw[red](0,.8)node{$\sigma$};\draw[NavyBlue](0,-.5)node{$\gamma_{k_i}$};
\draw(-1.5,-1)node{$\bullet$}(-1.1,0)node{$\bullet$}(1,2)node{$\bullet$};
\end{tikzpicture}
\caption{The intersections between $\sigma$ and $\T$}
\label{fig:int0}
\end{figure}
\item Each segment $\sigma(i-1,i)$ ($1\leq i\leq p$) of $\sigma$ corresponds to a graded arrow $a_i:=\alpha(\sigma(i-1,i))$ between $k_{i-1}$ and $k_{i}$ in $\overline{Q}$. 
Then we obtain a walk in $\overline{Q}$, called a string:
  \begin{gather}\label{eq:string}
  w(\sigma):\xymatrix{
  k_0\ar@{-}[r]^{a_1}&k_1\ar@{-}[r]^{a_2}&\cdots
  \ar@{-}[r]^{a_p}&k_p}.
  \end{gather}
We define $\epsilon(a_i)=+$ if $a_i$ points to the right, and $\epsilon(a_i)=-$ otherwise.
\item The string $w(\sigma)$ induces a graded $\EE_\T$-module $|X_\sigma|$ and a map $d_\sigma$ on $|X_\sigma|$ of degree $1$ as follows.
\begin{itemize}
\item $|X_\sigma|=\oplus_{i=0}^p S_{k_i}[\varrho_i],$
where $\varrho_0=0$ and $\varrho_{i}=\varrho_{i-1}-\epsilon(a_i)\deg a_i$ for $1\leq i\leq p$.
\item For each $a_i$, if $\epsilon(a_i)=+$, then the map $\pi_{a_i}:S_{k_{i-1}}\to S_{k_i}[1-\deg a_i]$ induces a component $S_{k_{i-1}}[\varrho_{i-1}]\to S_{k_i}[\varrho_i]$ of $d_\sigma$; if $\epsilon(a_i)=-$, then the map $\pi_{a_i}:S_{k_{i}}\to S_{k_{i-1}}[1-\deg a_i]$ induces a component $S_{k_{i}}[\varrho_{i}]\to S_{k_{i-1}}[\varrho_{i-1}]$ of $d_\sigma$. The other components of $d_\sigma$ are zero.
\end{itemize}
\end{itemize}
\end{construction}

\begin{proposition}
$X_\sigma:=(|X_\sigma|,d_\sigma)$ is a perfect dg $\EE_\T$-module in $\per(\EE_\T).$
\end{proposition}

\begin{proof}
We only need to prove that $d^2=0$. By Lemma~\ref{lem:comp}, this follows from that any two neighboring arc segments of $\sigma$ are from different triangles.
\end{proof}

By construction, for any oriented general closed arc $\sigma'$, if $\sigma'\sim\sigma$ then $X_{\sigma'}=X_{\sigma}$. Let $\overline{\sigma}$ be the oriented general closed arc obtained from $\sigma$ by conversing the orientation. It is easy to see $X_{\overline{\sigma}}\cong X_{\sigma}[l]$ for some $l$. Denote by $\widetilde{X}(\sigma)$ the shift orbit $X_\sigma[\mathbb{Z}]$ of $X_\sigma$. Then
\[\sigma\mapsto \widetilde{X}(\sigma)\]
is a well-defined map from the set $\GCA(\surfo)$ to the set of objects in the orbit category $\per\EE_\T/[1]$.

\subsection{Homomorphisms between strings}\label{app:3}
Let $\sigma$ be an oriented general closed arc as in Construction~\ref{cons:string}, with
the string \eqref{eq:string} and the associated dg $\EE_\T$-modules $X_\sigma$,
whose underly graded module is $\bigoplus_{i=0}^p S_{k_i}[\varrho_i]$.

Now, take another oriented general closed arc $\tau$ with $\Int_{\surf-\Tri}(\sigma,\tau)=0$ and
which is in a minimal position w.r.t. $\T$ and $\sigma$. Note that $\tau$ may be isotopic to $\sigma$.
Suppose that $\tau$ intersects $\T$ at $W_0,\ldots,W_q$ in order,
where $W_i$ is in the arc $\gamma_{j_i}\in\T$,
with staring point $W_{-1}$ and ending point $W_{q+1}$.
Then there are the associated string
\[w(\tau):\xymatrix{
	{j_0}\ar@{-}[r]^{{b_1}}&{j_1}\ar@{-}[r]^{{b_2}}&\cdots\ar@{-}[r]^{{b_q}}&{j_q}
},\]
where $b_i$ is the arrow in $\overline{Q}_\T$ induced by the segment $\tau(i-1,i)$,
and the associated dg $\EE_\T$-module $X_\tau$,
whose underly graded module is $\bigoplus_{i=0}^q S_{j_i}[\kappa_i]$.

\begin{construction}\label{cons:above}
Suppose $V_{-1}=W_{-1}$. Then there is an angle $\theta(\sigma,\tau)$ from $\sigma$ to $\tau$ clockwise at this decorating point (cf. Figure~\ref{fig:3cases}). We will construct an element  $\varphi(\sigma,\tau)$ in $\hom^0(X_\sigma,X_\tau[\upsilon])$ induced by $\theta(\sigma,\tau)$. Here the value of $\upsilon=\upsilon(\sigma,\tau)$ is determined by the relative position of the segments $\sigma(-1,0)$ and $\tau(-1,0)$. There are four cases shown in Figure~\ref{fig:3cases}, where $\upsilon=0,1,2,3$, respectively.
	
\begin{figure}[ht]\centering
\begin{tikzpicture}[yscale=.5,xscale=.6,rotate=0]
\draw[blue, thick] (.7,.7)node{$\sigma$} (-.7,.7)node{$\tau$};
\draw[blue, thick,>=stealth] (0,1.5)edge[->-=.5,>=stealth](-.8,-.5);
\draw[blue, thick,>=stealth] (0,1.5)edge[->-=.5,>=stealth](.8,-.5);
\draw[] (0,-.8)node{ };
		\draw[thick](-2,0)node{$\bullet$}to(2,0)node{$\bullet$}to(0,4)node{$\bullet$}to(-2,0)
	(0,1.5)node[white]{$\bullet$}node[red]{$\circ$};
\draw[red,thick] (-.26,.9)edge[bend right] (.26,.9);
	\draw[red](0,.5)node{$\theta$};
\end{tikzpicture}
\qquad
\begin{tikzpicture}[yscale=.5,xscale=.6,rotate=0]
\draw[blue, thick] (.4,.5)node{$\sigma$} (-.4,2.1)node{$\tau$};2
\draw[blue, thick,>=stealth] (0,1.5)edge[->-=.5,>=stealth](0,-.5);
\draw[blue, thick,>=stealth] (0,1.5)edge[->-=.5,>=stealth](-1.5,2.5);
\draw[red,thick] (0,1)edge[bend left] (-.5,1.85);
\draw[red](-.7,.8)node{$\theta$};
\draw[] (0,-.8)node{ };
	\draw[thick](-2,0)node{$\bullet$}to(2,0)node{$\bullet$}to(0,4)node{$\bullet$}to(-2,0)
	(0,1.5)node[white]{$\bullet$}node[red]{$\circ$};
\end{tikzpicture}
\qquad
\begin{tikzpicture}[yscale=.5,xscale=.6,rotate=0]
\draw[blue, thick] (.3,.5)node{$\sigma$} (.7,1.5)node{$\tau$};
\draw[blue, thick,>=stealth] (0,1.5)edge[->-=.5,>=stealth](0,-.5);
\draw[blue, thick,>=stealth] (0,1.5)edge[->-=.5,>=stealth](1.5,2.5);
\draw[red,thick] (0,1)..controls +(160:1) and +(120:1)..(.4,1.8);
\draw[red](-.7,1.2)node{$\theta$};
\draw[] (0,-.8)node{ };
	\draw[thick](-2,0)node{$\bullet$}to(2,0)node{$\bullet$}to(0,4)node{$\bullet$}to(-2,0)
	(0,1.5)node[white]{$\bullet$}node[red]{$\circ$};
\end{tikzpicture}
\qquad
\begin{tikzpicture}[yscale=.5,xscale=.6,rotate=0]
\draw[blue, thick] (-.7,.7)node{$\sigma$} (.7,.7)node{$\tau$};
\draw[blue, thick,>=stealth] (0,1.5)edge[->-=.5,>=stealth](-.8,-.5);
\draw[blue, thick,>=stealth] (0,1.5)edge[->-=.5,>=stealth](.8,-.5);
\draw[red,thick] (-.2,1)..controls +(120:1.5) and +(60:1.5)..(.2,1);
\draw[red](0,2.5)node{$\theta$};
\draw[] (0,-.8)node{ };
	\draw[thick](-2,0)node{$\bullet$}to(2,0)node{$\bullet$}to(0,4)node{$\bullet$}to(-2,0)
(0,1.5)node[white]{$\bullet$}node[red]{$\circ$};
\end{tikzpicture}
\caption{The four cases for the starting segments of $\sigma$ and $\tau$}
\label{fig:3cases}
\end{figure}
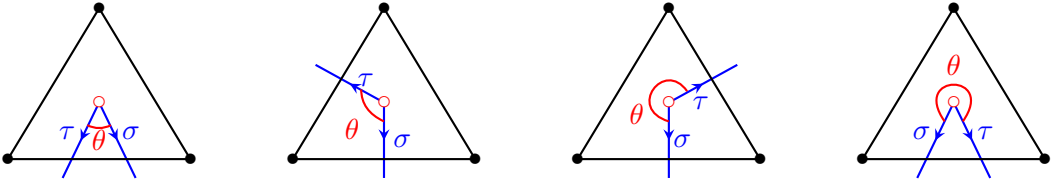

Note that there is a unique $s\geq 0$ with a unique segment $l(\sigma,\tau)$ in the triangle $\Lambda_s$ connecting $V_s$ and $W_s$ such that $l(\sigma,\tau)$, $\sigma(-1,s)$ and  $\tau(-1,s)$ enclose a contractible triangle having $\theta(\sigma,\tau)$ as an internal angle. (A degenerate case is that when $\sigma\sim\tau$, $l(\sigma,\tau)$ is the decorating point in $\Lambda_{p+1}$.) It is clear that $s=0$ for the last three cases in Figure~\ref{fig:3cases} and $s>0$ for the first case. We show in Figure~\ref{fig:lastcase} all the possible subcases for $\Lambda_s$ when $s>0$.
	
When the associated graded arrow $\alpha\left(l(\sigma,\tau)\right)$ exists and is from $\nu_{V_s}$ to $\nu_{W_s}$, let $\varphi_s=\pi_{\alpha(l(\sigma,\tau))}:S_{k_s}\to S_{j_s}[\deg\varphi_s]$. By construction, for $i<s$, we have $S_{k_i}=S{j_i}$ and let $\varphi_{i}=\id:S_{k_i}\to S_{j_i}$. We construct $\varphi(\sigma,\tau)$ in $\hom^0(X_\sigma,X_\tau[\upsilon])$ whose nonzero components are $\varphi_i[\varrho_i]$. That is, when $s=0$ (i.e. the last three cases in Figure~\ref{fig:3cases}), $\varphi(\sigma,\tau)$ has the following form
	\begin{equation}\label{eq:form}
	\xymatrix@C=6pc{S_{k_0}\ar[d]_{\varphi_0}\ar@{-}[r]^{\pi_{a_{1}}\quad}&S_{k_{1}}[\varrho_{1}]\ar@{-}[r]&\cdots\\
		S_{j_0}[\deg\varphi_0]\ar@{-}[r]^{(-1)^{\deg\varphi_0}\pi_{b_{1}}\quad}&S_{j_{1}}[\kappa_{1}+\deg\varphi_0]\ar@{-}[r]&\cdots
	}
	\end{equation}
	and when $s>0$ (i.e. the first case in Figure~\ref{fig:3cases}), $\varphi(\sigma,\tau)$ has the following form
\begin{equation}\label{eq:form2} \xymatrix@C=3pc{S_{k_0}\ar@{=}[d]\ar@{-}[r]&\cdots\ar@{-}[r]&S_{k_{s-1}}[\varrho_{s-1}]
    \ar@{=}[d]\ar@{-}[r]^{\quad\pi_{a_s}[\varrho_{s-1}]}&S_{k_s}[\varrho_s]\ar@{-->}[d]^{\varphi_s[\varrho_s]}\ar@{-}[r]^{\pi_{a_{s+1}}[\varrho_{s}]\quad}&S_{k_{s+1}}[\varrho_{s+1}]\ar@{-}[r]&\cdots\\
S_{k_0}\ar@{-}[r]&\cdots\ar@{-}[r]&S_{j_{s-1}}[\kappa_{s-1}]\ar@{-}[r]^{\quad\pi_{b_s}[\kappa_{{s-1}}]}&S_{j_s}[\kappa_s]\ar@{-}[r]^{\pi_{b_{s+1}}[\kappa_{s}]\quad}&S_{j_{s+1}}[\kappa_{s+1}]\ar@{-}[r]&\cdots
	}
\end{equation}
where $\varphi_s[\varrho_s]$ exists if and only if one of the cases in (a) or (b) in Figure~\ref{fig:lastcase} occurs.
	
\end{construction}

\begin{figure}[ht]\centering
\begin{tikzpicture}[yscale=.45,xscale=-.6,rotate=180]
\draw[] (0,3)node{(a)};
\draw[] (0,5)node{ };
\end{tikzpicture}
\qquad
\begin{tikzpicture}[yscale=.45,xscale=-.6,rotate=180]
\draw[thick](-2,0)node{$\bullet$}to(2,0)node{$\bullet$}to(0,4)node{$\bullet$}to(-2,0)
    (0,1.5)node[red]{$\circ$};
\draw[blue, thick] (.1,1)node{$\tau$} (1.1,.7)node{$\sigma$};
\draw[blue, thick,>=stealth] (0,-.5)edge[bend right,->-=.5,>=stealth](-.7,3.7);
\draw[blue, thick,>=stealth] (.5,-.5)edge[bend left=10,->-=.5,>=stealth](1.5,2);
\draw[red, thick,>=stealth] (-.35,3.3)edge[bend right=10,>=stealth](1.25,1.5);
\draw[red, thick] (-.4,3.5)node[below]{$B$} (1.1,2.1)node[below]{$A$};
\end{tikzpicture}
\qquad
\begin{tikzpicture}[yscale=.45,xscale=-.6,rotate=180]
	\draw[thick](-2,0)node{$\bullet$}to(2,0)node{$\bullet$}to(0,4)node{$\bullet$}to(-2,0)
	(0,1.5)node[red]{$\circ$};
\draw[blue, thick] (.1,1)node{$\tau$} (1.1,.7)node{$\sigma$};
\draw[blue, thick,->-=.3,>=stealth] (0,-.5).. controls +(70:3.5) and
    +(90:4)..(-.7,-.5);    
\draw[blue, thick,>=stealth] (.5,-.5)edge[bend left=10,->-=.5,>=stealth](1.5,2);            
\draw[red, thick,>=stealth] (-.7,0)..controls +(+100:2) and +(120:3)..(1.2,1.6);        
\draw[red, thick] (-1.3,0)node[above]{$B$} (1.3,2)node[below]{$A$};
\end{tikzpicture}
\qquad
\begin{tikzpicture}[yscale=.45,xscale=-.6,rotate=180]
	\draw[thick](-2,0)node{$\bullet$}to(2,0)node{$\bullet$}to(0,4)node{$\bullet$}to(-2,0)
	(0,1.5)node[red]{$\circ$};
\draw[blue, thick] (.1,1)node{$\tau$} (1.1,.7)node{$\sigma$};
	\draw[blue, thick,->-=.3,>=stealth] (0,-.5).. controls +(70:3.5) and
    +(90:4)..(-.7,-.5);    
\draw[blue, thick,>=stealth] (.8,-.5)edge[bend right,->-=.5,>=stealth](-.7,3.7);            
\draw[red, thick,>=stealth] (-.7,0)edge[bend left=10,>=stealth](-.3,3.4);        
\draw[red, thick] (-1.3,0)node[above]{$B$} (-.9,2.3)node[below]{$A$};
\end{tikzpicture}
	
\begin{tikzpicture}[yscale=.6,xscale=.8,rotate=180]
\draw[] (0,3)node{(b)};
\draw[] (0,5)node{ };
\end{tikzpicture}
\qquad
\begin{tikzpicture}[yscale=.45,xscale=.6,rotate=180]
	\draw[thick](-2,0)node{$\bullet$}to(2,0)node{$\bullet$}to(0,4)node{$\bullet$}to(-2,0)
	(0,1.5)node[red]{$\circ$};
\draw[blue, thick] (.1,1)node{$\sigma$} (1.1,.7)node{$\tau$};
\draw[blue, thick,>=stealth] (0,-.5)edge[bend right,->-=.5,>=stealth](-.7,3.7);
\draw[blue, thick,>=stealth] (.5,-.5)edge[bend left=10,->-=.5,>=stealth](1.5,2);
\draw[red, thick,>=stealth] (-.25,3.5)edge[bend right=20,>=stealth](.85,2.3);
\draw[red, thick] (-.3,3.5)node[below]{$A$} (1.2,2.1)node[below]{$B$};
\end{tikzpicture}
\qquad
\begin{tikzpicture}[yscale=.45,xscale=.6,rotate=180]
	\draw[thick](-2,0)node{$\bullet$}to(2,0)node{$\bullet$}to(0,4)node{$\bullet$}to(-2,0)
	(0,1.5)node[red]{$\circ$};
\draw[blue, thick] (.1,1)node{$\sigma$} (1.1,.7)node{$\tau$};
\draw[blue, thick,->-=.3,>=stealth] (0,-.5).. controls +(70:3.5) and
    +(90:4)..(-.7,-.5);    
\draw[blue, thick,>=stealth] (.5,-.5)edge[bend left=10,->-=.5,>=stealth](1.5,2);            
\draw[red, thick,>=stealth] (-1.3,0)edge[bend left=20,>=stealth](.4,3.2);        
\draw[red, thick] (-1.3,0)node[above]{$A$} (.6,3)node[below]{$B$};
\end{tikzpicture}
	\qquad
	\begin{tikzpicture}[yscale=.45,xscale=.6,rotate=180]
	\draw[thick](-2,0)node{$\bullet$}to(2,0)node{$\bullet$}to(0,4)node{$\bullet$}to(-2,0)
	(0,1.5)node[red]{$\circ$};
	\draw[blue, thick] (.1,1)node{$\sigma$} (1.1,.7)node{$\tau$};
	\draw[blue, thick,->-=.3,>=stealth] (0,-.5).. controls +(70:3.5) and +(90:4)..(-.7,-.5);    
	\draw[blue, thick,>=stealth] (.8,-.5)edge[bend right,->-=.5,>=stealth](-.7,3.7);            
	\draw[red, thick,>=stealth] (-1.3,0)edge[bend right=20,>=stealth](-.85,2.3);        
	\draw[red, thick] (-1.3,0)node[above]{$A$} (-.9,2.3)node[below]{$B$};
\end{tikzpicture}
	
\begin{tikzpicture}[yscale=.6,xscale=.8,rotate=180]
\draw[] (0,3)node{(c)};
\draw[] (0,5)node{ };
\end{tikzpicture}
\qquad
\begin{tikzpicture}[yscale=.45,xscale=.6,rotate=180]
	\draw[thick](-2,0)node{$\bullet$}to(2,0)node{$\bullet$}to(0,4)node{$\bullet$}to(-2,0)
	(0,1.5)node[red]{$\circ$};
\draw[blue, thick] (-1.1,.7)node{$\sigma$} (1.1,.7)node{$\tau$};
\draw[blue, thick,>=stealth] (-.5,-.5)edge[bend right=10,->-=.5,>=stealth](-1.5,2);
\draw[blue, thick,>=stealth] (.5,-.5)edge[bend left=10,->-=.5,>=stealth](1.5,2);
\end{tikzpicture}
\qquad
\begin{tikzpicture}[yscale=.45,xscale=.6,rotate=180]
	\draw[thick](-2,0)node{$\bullet$}to(2,0)node{$\bullet$}to(0,4)node{$\bullet$}to(-2,0)
	(0,1.5)node[red]{$\circ$};
\draw[blue, thick] (-1.1,1)node{$\sigma$} (1.1,.7)node{$\tau$};
\draw[blue, thick,>=stealth] (-.8,-.5)edge[bend left,->-=.5,>=stealth](.7,3.7);    
\draw[blue, thick,>=stealth] (.5,-.5)edge[bend left=10,->-=.5,>=stealth](1.5,2);            
\end{tikzpicture}
\qquad
\begin{tikzpicture}[yscale=.45,xscale=.6,rotate=180]
	\draw[thick](-2,0)node{$\bullet$}to(2,0)node{$\bullet$}to(0,4)node{$\bullet$}to(-2,0)
	(0,1.5)node[red]{$\circ$};
\draw[blue, thick] (-1.1,.7)node{$\sigma$} (1.1,.7)node{$\tau$};
\draw[blue, thick,>=stealth] (-.5,-.5)edge[bend right=10,->-=.5,>=stealth](-1.5,2);
\draw[blue, thick,>=stealth] (.8,-.5)edge[bend right,->-=.5,>=stealth](-.7,3.7);
\end{tikzpicture}
	
\begin{tikzpicture}[yscale=.6,xscale=.8,rotate=180]
\draw[] (0,3)node{(d)};
\draw[] (0,5)node{ };
\end{tikzpicture}
\qquad
\begin{tikzpicture}[yscale=.45,xscale=.6,rotate=180]
\draw[blue, thick] (-1.1,.7)node{$\sigma$} (.6,.7)node{$\tau$};
\draw[blue, thick,>=stealth] (-.5,-.5)edge[bend right=10,->-=.5,>=stealth](-1.5,2); 
\draw[blue, thick,>=stealth] (.5,-.5)edge[bend right=10,->-=.5,>=stealth](0,1.5);
	\draw[thick](-2,0)node{$\bullet$}to(2,0)node{$\bullet$}to(0,4)node{$\bullet$}to(-2,0)
	(0,1.5)node[white]{$\bullet$}node[red]{$\circ$};
\end{tikzpicture}
\qquad
\begin{tikzpicture}[yscale=.45,xscale=.6,rotate=180]
\draw[blue, thick] (-1.1,1)node{$\sigma$} (.6,.7)node{$\tau$};
\draw[blue, thick,>=stealth] (-.8,-.5)edge[bend left,->-=.5,>=stealth](.7,3.7);    
\draw[blue, thick,>=stealth] (.5,-.5)edge[bend right=10,->-=.5,>=stealth](0,1.5);
	\draw[thick](-2,0)node{$\bullet$}to(2,0)node{$\bullet$}to(0,4)node{$\bullet$}to(-2,0)
	(0,1.5)node[white]{$\bullet$}node[red]{$\circ$};
\end{tikzpicture}
\qquad
\begin{tikzpicture}[yscale=.45,xscale=.6,rotate=180]
\draw[blue, thick] (-1,.7)node{$\sigma$} (.6,.7)node{$\tau$};
\draw[blue, thick,->-=.3,>=stealth] (-1,-.5).. controls +(80:4) and +(90:3.5) ..(1,-.5); 
\draw[blue, thick,>=stealth] (.5,-.5)edge[bend right=10,->-=.5,>=stealth](0,1.5);
	\draw[thick](-2,0)node{$\bullet$}to(2,0)node{$\bullet$}to(0,4)node{$\bullet$}to(-2,0)
	(0,1.5)node[white]{$\bullet$}node[red]{$\circ$};
\end{tikzpicture}
	
\begin{tikzpicture}[yscale=.6,xscale=.8,rotate=180]
\draw[] (0,3)node{(e)};
\draw[] (0,5)node{ };
\end{tikzpicture}
\qquad
\begin{tikzpicture}[yscale=.45,xscale=-.6,rotate=180]
\draw[blue, thick] (-1.1,.7)node{$\tau$} (.6,.7)node{$\sigma$};
\draw[blue, thick,>=stealth] (-.5,-.5)edge[bend right=10,->-=.5,>=stealth](-1.5,2); 
\draw[blue, thick,>=stealth] (.5,-.5)edge[bend right=10,->-=.5,>=stealth](0,1.5);
	\draw[thick](-2,0)node{$\bullet$}to(2,0)node{$\bullet$}to(0,4)node{$\bullet$}to(-2,0)
	(0,1.5)node[white]{$\bullet$}node[red]{$\circ$};
\end{tikzpicture}
\qquad
\begin{tikzpicture}[yscale=.45,xscale=-.6,rotate=180]
\draw[blue, thick] (-1.1,1)node{$\tau$} (.6,.7)node{$\sigma$};
\draw[blue, thick,>=stealth] (-.8,-.5)edge[bend left,->-=.5,>=stealth](.7,3.7);    
\draw[blue, thick,>=stealth] (.5,-.5)edge[bend right=10,->-=.5,>=stealth](0,1.5);
	\draw[thick](-2,0)node{$\bullet$}to(2,0)node{$\bullet$}to(0,4)node{$\bullet$}to(-2,0)
	(0,1.5)node[white]{$\bullet$}node[red]{$\circ$};
\end{tikzpicture}
\qquad
\begin{tikzpicture}[yscale=.45,xscale=-.6,rotate=180]
\draw[blue, thick] (-1,.7)node{$\tau$} (.6,.7)node{$\sigma$};
\draw[blue, thick,->-=.3,>=stealth] (-1,-.5).. controls +(80:4) and +(90:3.5)
    ..(1,-.5); 
\draw[blue, thick,>=stealth] (.5,-.5)edge[bend right=10,->-=.5,>=stealth](0,1.5);
	\draw[thick](-2,0)node{$\bullet$}to(2,0)node{$\bullet$}to(0,4)node{$\bullet$}to(-2,0)
	(0,1.5)node[white]{$\bullet$}node[red]{$\circ$};
\end{tikzpicture}
	
\begin{tikzpicture}[yscale=.6,xscale=.8,rotate=180]
\draw[] (0,3)node{(f)};
\draw[] (0,5)node{ };
\end{tikzpicture}
\qquad
\begin{tikzpicture}[yscale=.45,xscale=-.6,rotate=180]
\draw[blue, thick] (-.8,.7)node{$\tau$} (.7,.7)node{$\sigma$};
\draw[blue, thick,>=stealth] (-.5,-.5)edge[bend left=10,->-=.5,>=stealth](0,1.5);
\draw[blue, thick,>=stealth] (.5,-.5)edge[bend right=10,->-=.5,>=stealth](0,1.5);
	\draw[thick](-2,0)node{$\bullet$}to(2,0)node{$\bullet$}to(0,4)node{$\bullet$}to(-2,0)
	(0,1.5)node[white]{$\bullet$}node[red]{$\circ$};
\end{tikzpicture}
\qquad

\caption{Relative positions of $\sigma(s-1,s)$ and $\tau(s-1,s)$ in the triangle $\Lambda_s$}
\label{fig:lastcase}
\end{figure}
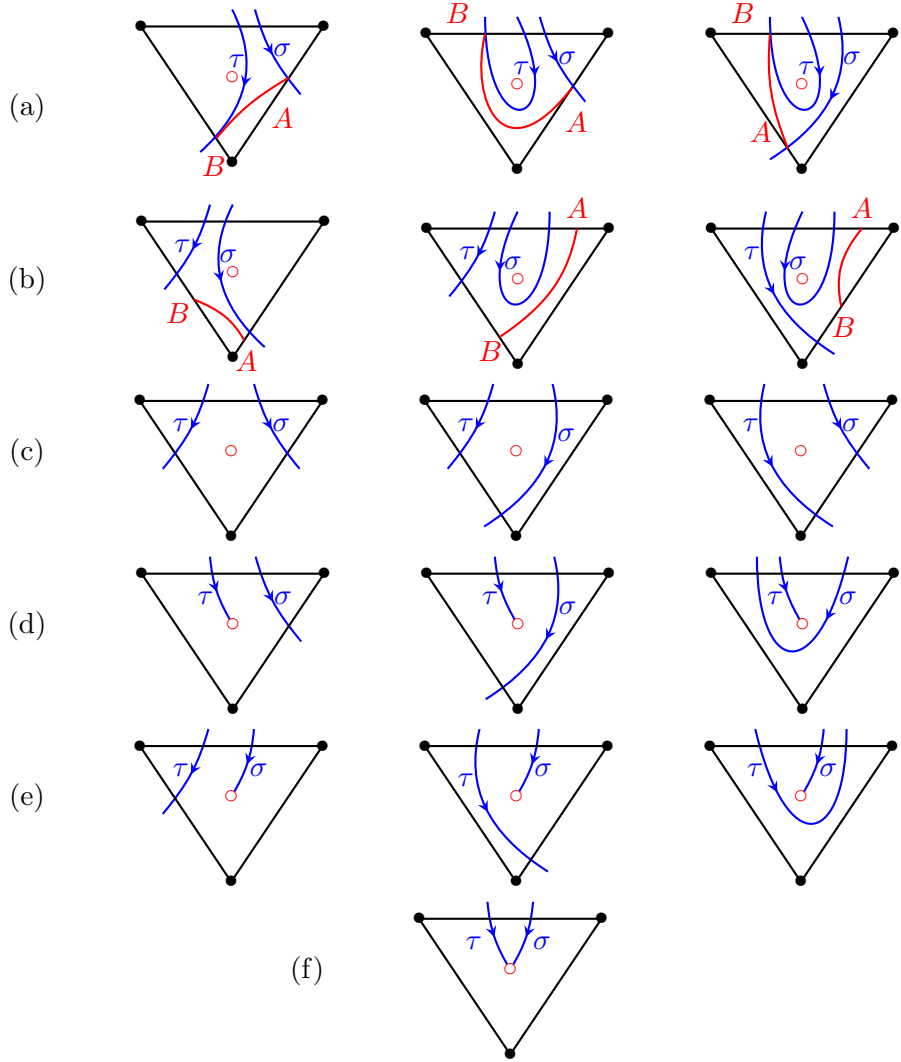

\begin{lemma}
$\varphi(\sigma,\tau)$ is in $Z^0 \hom(X_\sigma,X_\tau[\upsilon])$.
\end{lemma}

\begin{proof}
It suffices to prove that the components of $\varphi(\sigma,\tau)$ commute with the differentials of $X_\sigma$ and $X_\tau$. Since $\varphi_s$ is not from the same triangle as $\pi_{a_{s+1}}$ or $\pi_{b_{s+1}}$, their compositions (if exist) are zero. Hence we only need to prove that when $s>0$, $\varphi_s$ commutes with $\pi_{a_s}$ and $\pi_{b_s}$ in a suitable way. Consider the cases for $\Lambda_s$ in Figure~\ref{fig:lastcase}.
\begin{description}
\item[Cases in (a)]  $\epsilon(a_{s})=\epsilon(b_{s})=+$ and $\pi_{b_{s}}=\varphi_s\pi_{a_{s}}$, so $\varphi(\sigma,\tau)\in Z^0\hom(X_\sigma,X_\tau)$;
\item[Cases in (b)] $\epsilon(a_{s})=\epsilon(b_{s})=-$ and $\pi_{a_{s}}=\pi_{b_{s}}\varphi_s$, so $\varphi(\sigma,\tau)\in Z^0\hom(X_\sigma,X_\tau)$;
\item[Cases in (c), (d), (e) or (f)] $\epsilon(\pi_{a_{s}})=+, \epsilon(\pi_{b_{s}})=-$ (if exist) and $\varphi_s$ does not exist, so $\varphi(\sigma,\tau)\in Z^0\hom(X_\sigma,X_\tau)$.
\end{description}
\end{proof}

\begin{lemma}
$\varphi(\sigma,\tau)$ is not null-homotopic.
\end{lemma}

\begin{proof}
For the first case in Figure~\ref{fig:3cases}, the identities in the from \eqref{eq:form2} do not factor through $\pi_\alpha$ for any graded arrow $\alpha$ in $\overline{Q}$. Hence $\varphi(\sigma,\tau)$ is not null-homotopic.

For the second and third cases in Figure~\ref{fig:3cases}, since $\varphi_0$ in the from \eqref{eq:form} is of degree 1 or 2 and is not from the same triangle as $a_1$ or $b_1$, it does not factor through $\pi_{a_1}$ or $\pi_{b_1}$. Hence $\varphi(\sigma,\tau)$ is not null-homotopic.

Assume that $\varphi(\sigma,\tau)$ is null-homotopic in the last case in Figure~\ref{fig:3cases}. Then there exist morphisms $\psi_{u,v}:S_{k_u}\to S_{j_v}[\kappa_v+2]$ such that $\varphi_0=\psi_{2,1}\circ\pi_{a_1}+\pi_{b_{1}}\circ\psi_{1,2}$ and $\pi_{a_i}\circ\psi_{i+1,i}+\psi_{i,i+1}\circ\pi_{b_{i}}+\psi_{i+2,i+1}\circ\pi_{a_{i+1}}+\pi_{b_{i+1}}\circ\psi_{i+1,i+2}=0$ for $i\geq1$. Let $t$ be the maximal integer such that $a_i=b_i$ for $i< t$. Since $\deg\varphi_0=3$, by Lemma~\ref{lem:comp} repeatedly, the morphisms $\pi_{a_i}\circ\psi_{i+1,i}+\psi_{i,i+1}\circ\pi_{b_{i}}$ are also nonzero and of degree 3, for $i< t$. Note that $a_t$ and $b_t$ are from the triangle $\Lambda_t$. All possible cases for $\Lambda_t$ are shown in Figure~\ref{fig:lastcase}, where $s$ should be replaced by $t$ and $\sigma,\tau$ should be switched each other. It is checked case by case in the following that there is a contradiction. Hence $\varphi(\sigma,\tau)$ is null-homotopic.
\begin{description}
  \item[\textbf{Cases in (a) or (b)}]
  We have $\epsilon(\pi_{a_{t}})=\epsilon(\pi_{b_{t}})$ but $\deg\pi_{a_{t}}\neq\deg\pi_{a_{t}}$. By Lemma~\ref{lem:comp}, we have that $\pi_{a_t}\circ\psi_{t+1,t}+\psi_{t,t+1}\circ\pi_{b_{t}}$ is nonzero and of degree less than 3. Since $a_{t+1}$ and $b_{t+1}$ are not from $\Lambda_t$, by Lemma~\ref{lem:comp}, there are no $\psi_{t+2,t+1}$ and $\psi_{t+1,t+2}$ satisfying $\pi_{a_t}\circ\psi_{t+1,t}+\psi_{t,t+1}\circ\pi_{b_{t}}+\psi_{t+2,t+1}\circ\pi_{a_{t+1}}+\pi_{b_{t+1}}\circ\psi_{t+1,t+2}=0$. This is a contradiction.
 \item[\textbf{Cases in (c), (d), (e) or (f)}] We have $\epsilon(\pi_{a_t})=-$ and $\epsilon(\pi_{b_t})=+$, if exist.
       Then $\pi_{a_{t-1}}\circ\psi_{t,t-1}+\psi_{t-1,t}\circ\pi_{b_{t-1}}$  does not factor through $\pi_{a_t}$ or $\pi_{b_t}$ because of the directions of the maps. This is a contradiction.
\end{description}
\end{proof}

Combining the above two lemmas, we have the following result.

\begin{proposition}\label{defpp}
Let $\sigma,\tau$ be oriented general closed arcs in $\surfo$ with $\Int_{\surf-\Tri}(\sigma,\tau)=0$ and whose starting points coincide. Then (the homotopy class of) $\varphi(\sigma,\tau)$ is a non-zero morphism in $\Hom_{\per\EE_\T}(X_\sigma,X_\tau[\upsilon])$.
\end{proposition}

In particular, $\varphi(\sigma,\tau)$ can be regarded as a morphism in $\Hom_{\per\EE_\T/[1]}(\widetilde{X}(\sigma),\widetilde{X}(\tau))$.

\begin{corollary}\label{cor:comp}
Let $\sigma_1,\sigma_2,\sigma_3$ be oriented general closed arcs in $\surfo$ with $\Int_{\surf-\Tri}(\sigma_i,\sigma_j)=0$ for any $i,j$ and which share the same starting point.
If the start segments of $\sigma_1,\sigma_2$ and $\sigma_3$ are in clockwise order at the starting point, then
\begin{eqnarray}\label{eq:neq0}
\varphi(\sigma_2,\sigma_{3})\circ\varphi(\sigma_{1},\sigma_{2})=\varphi(\sigma_{1},\sigma_{3}).
\end{eqnarray}
\end{corollary}
\begin{figure}[ht]\centering
\quad\quad\begin{tikzpicture}[yscale=.4,xscale=.5,rotate=0]
\draw[red,thick](-.5,-.8)node[below,black]{$\sigma_2$}to[bend right=7](0,1.5)
(.5,-.8)node[below,black]{$\sigma_1$}to[bend right=7](0,1.5)
(-1.5,-.8)node[below,black]{$\sigma_3$}to[bend right=7](0,1.5);
\draw[NavyBlue, thick](-2,0)node{$\bullet$}to(2,0)node{$\bullet$}to(0,4)node{$\bullet$}to(-2,0)
    (0,1.5)node[white]{$\bullet$}node[red]{$\circ$};
\end{tikzpicture}
\quad
\begin{tikzpicture}[yscale=.4,xscale=.5,rotate=180]
\draw[red,thick](-2,2.5)node[right,black]{$\sigma_1$}to[bend right=7](0,1.5)
(2,2.5)node[left,black]{$\sigma_3$}to[bend right=7](0,1.5)
(1.5,3.5)node[left,black]{$\sigma_2$}to[bend right=7](0,1.5);
\draw[NavyBlue, thick](-2,0)node{$\bullet$}to(2,0)node{$\bullet$}to(0,4)node{$\bullet$}to(-2,0)
    (0,1.5)node[white]{$\bullet$}node[red]{$\circ$} (0,2.5)node[below,black]{$ $};
\end{tikzpicture}
\quad
\begin{tikzpicture}[yscale=.4,xscale=.5,rotate=0]
\draw[red,thick](-2,2.5)node[left,black]{$\sigma_3$}to[bend right=7](0,1.5)
(2,2.5)node[right,black]{$\sigma_1$}to[bend right=7](0,1.5)
(-2.5,1)node[left,black]{$\sigma_2$}to[bend right=7](0,1.5);
\draw[NavyBlue, thick](-2,0)node{$\bullet$}to(2,0)node{$\bullet$}to(0,4)node{$\bullet$}to(-2,0)
    (0,1.5)node[white]{$\bullet$}node[red]{$\circ$};
\end{tikzpicture}

\begin{tikzpicture}[yscale=.4,xscale=.5,rotate=0]
\draw[red,thick](-2,2.5)node[left,black]{$\sigma_3$}to[bend right=7](0,1.5)
(2,2.5)node[right,black]{$\sigma_1$}to[bend right=7](0,1.5)
(0,-1.5)node[right,black]{$\sigma_2$}to[bend right=7](0,1.5);
\draw[NavyBlue, thick](-2,0)node{$\bullet$}to(2,0)node{$\bullet$}to(0,4)node{$\bullet$}to(-2,0)
    (0,1.5)node[white]{$\bullet$}node[red]{$\circ$};
\end{tikzpicture}
\quad
\begin{tikzpicture}[yscale=.4,xscale=.5,rotate=180]
\draw[red,thick](-.5,-.8)node[above,black]{$\sigma_1$}to[bend right=7](0,1.5)
(.5,-.8)node[above,black]{$\sigma_3$}to[bend right=7](0,1.5)
(1.5,-.8)node[above,black]{$\sigma_2$}to[bend right=7](0,1.5);
\draw[NavyBlue, thick](-2,0)node{$\bullet$}to(2,0)node{$\bullet$}to(0,4)node{$\bullet$}to(-2,0)
    (0,1.5)node[white]{$\bullet$}node[red]{$\circ$};
\end{tikzpicture}
\quad\quad\quad
\begin{tikzpicture}[yscale=.4,xscale=.5,rotate=180]
\draw[red,thick](-.5,-.8)node[above,black]{$\sigma_1$}to[bend right=7](0,1.5)
(.5,-.8)node[above,black]{$\sigma_3$}to[bend right=7](0,1.5)
(2,2)node[left,black]{$\sigma_2$}to[bend right=7](0,1.5);
\draw[NavyBlue, thick](-2,0)node{$\bullet$}to(2,0)node{$\bullet$}to(0,4)node{$\bullet$}to(-2,0)
    (0,1.5)node[white]{$\bullet$}node[red]{$\circ$};
\end{tikzpicture}
\caption{The relative position of $\sigma_i$}
\label{fig:C}
\end{figure}
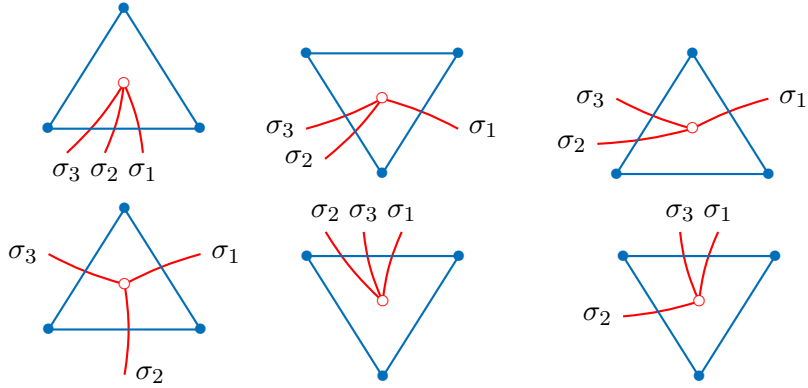
\begin{proof}
Consider the relative position of the first segments of $\sigma_i$. See Figure~\ref{fig:C} for all essential cases (up to mirror).
Then it is straightforward to check that $\varphi(\sigma_2,\sigma_3)\circ\varphi(\sigma_1,\sigma_2)$
is of the type in Construction~\ref{cons:above}. Hence we are done.
\end{proof}

\subsection{The induced Triangles}

Throughout this subsection, let $\sigma,\tau$ be oriented general closed arcs in $\surfo$ with $\Int_{\surf-\Tri}(\sigma,\tau)=0$. Suppose that $\sigma$ and $\tau$ share the same starting point and do not coincide in $\GCA(\surfo)$.

\begin{definition}
The \emph{(positive) extension} $\tau\wedge\sigma$ of $\tau$ by $\sigma$ (w.r.t. the common starting point) is defined in Figure~\ref{fig:htwist}.

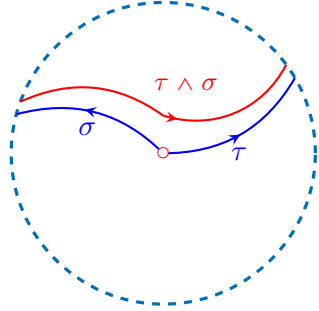
\begin{figure}[ht]\centering
\begin{tikzpicture}[scale=1]
\draw[NavyBlue,dashed,very thick](0,0)circle(2);
\draw[thick,blue](-195:2)edge[bend left,-<-=.5,>=stealth](0,0)
(0,0)edge[bend right,->-=.5,>=stealth](30:2)
    (1,0)node{$\tau$}(-1,.1)node[above]{$\sigma$};
\draw[thick,red](160:2)to[bend left](0,.5)(0.3,.7)
    node[above]{\small{$\tau\wedge\sigma$}};
\draw[thick,red](0,0.5)edge[bend right=40,->-=.1,>=stealth](36:2);
\draw(0,0)node[white] {$\bullet$} node[red](a){$\circ$};
\end{tikzpicture}
\caption{The extension}
\label{fig:htwist}
\end{figure}
\end{definition}

\begin{proposition}\label{prop:key}
There exists a non-trivial triangle in $\per\EE_\T$, whose image in $\per\EE_\T/[1]$ is
\[\widetilde{X}(\tau\wedge\sigma)\xrightarrow{\varphi(\tau\wedge\sigma,\overline{\sigma})}
\widetilde{X}(\sigma)\xrightarrow{\varphi(\sigma,\tau)}\widetilde{X}(\tau)\xrightarrow{\varphi(\overline{\tau},\overline{\tau\wedge\sigma})} \widetilde{X}(\tau\wedge\sigma).\]
\end{proposition}

\begin{proof}

Keep the notations for $\sigma$ and $\tau$ in the previous subsection. Using homological algebra, the mapping cone of $\varphi(\sigma,\tau)$ is the dg $\EE_\T$-module associated to the string arising from ${\tau\wedge\sigma}$. Hence we have the required triangle.
\end{proof}


\begin{thebibliography}{99}
\newcommand{\au}[1]{\textrm{#1},}
\newcommand{\ti}[1]{\textrm{#1},}
\newcommand{\jo}[1]{\textit{#1}}
\newcommand{\vo}[1]{\textbf{#1}}
\newcommand{\yr}[1]{(#1)}
\newcommand{\pp}[2]{#1--#2.}
\newcommand{\arxiv}[1]{\href{http://arxiv.org/abs/#1}{arXiv:#1}}
%
\bibitem{AI}
  T.~Aihara; O.~Iyama.
  Silting mutation in triangulated categories.
  \emph{J. Lond. Math. Soc. (2) }85 (2012), no. 3, 633--668.
  (\arxiv{1009.3370})

\bibitem{A}
  C.~Amiot.
  Cluster categories for algebras of global dimension 2 and quivers with potential.
  \emph{Ann. Inst. Fourier (Grenoble) }59 (2009), no. 6, 2525--2590.
  (\arxiv{0805.1035})
%
\bibitem{BS}
  T.~Bridgeland; I.~Smith.
  Quadratic differentials as stability conditions.
  \emph{Publ. Math. Inst. Hautes \'{E}tudes Sci. }  121 (2015), 155--278.
  (\arxiv{1302.7030})

%
\bibitem{BZ}
  T.~Br\"{u}stle; J.~Zhang.
  On the cluster category of a marked surface without punctures.
  \emph{Algebra Number Theory }5 (2011), no. 4, 529--566.
  (\arxiv{1005.2422})
%
\bibitem{BMRRT}
  A.~B.~Buan; R.~Marsh; M.~Reineke; I.~Reiten; G.~Todorov.
  Tilting theory and cluster combinatorics.
  \emph{Adv. Math. }204 (2006), no. 2, 572--618.
  (\arxiv{math/0402054})
%
\bibitem{DWZ}
  M.~Derksen; J.~Weyman; A.~Zelevinsky.
  Quivers with potentials and their representations. I. Mutations.
  \emph{Selecta Math. (N.S.) }14 (2008), no. 1, 59--119.
  (\arxiv{0704.0649})
%
\bibitem{G}
  V.~Ginzburg.
  Calabi-Yau algebras.
  \arxiv{math/0612139}.
%
\bibitem{FST}
  S.~Fomin; M.~Shapiro; D.~Thurston.
  Cluster algebras and triangulated surfaces. I. Cluster complexes.
  \emph{Acta Math. }201 (2008), no. 1, 83--146.
  (\arxiv{math/0608367})
%
\bibitem{FZ}
  S.~Fomin; A.~Zelevinsky.
  Cluster algebras. I. Foundations.
  \emph{J. Amer. Math. Soc. }15 (2002), no. 2, 497--529.
  (\arxiv{math/0104151})
%
\bibitem{GTW}
  A.~Gadbled; A-L.~Thiel; E.~Wagner.
  Categorical action of the extended braid group of affine type $A$.
  \emph{Commun. Contemp. Math. }Online (2016).
  (\arxiv{1504.07596})
%
\bibitem{Ke}
  B.~Keller.
  Deformed Calabi-Yau completions. With an appendix by Michel Van den Bergh.
  \emph{J. Reine Angew. Math. }654 (2011), 125--180.
  (\arxiv{0908.3499})
%
\bibitem{KN}
  B.~Keller; P.~Nicol\'{a}s.
  Weight structures and simple dg modules for positive dg algebras.
  \emph{Int. Math. Res. Not. IMRN }2013, no. 5, 1028--1078.
  (\arxiv{1009.5904})
%
\bibitem{KY}
  B.~Keller; D.~Yang.
  Derived equivalences from mutations of quivers with potential.
  \emph{Adv. Math. }226 (2011), no. 3, 2118--2168.
  (\arxiv{0906.0761})
%
\bibitem{KQ}
  A.~King; Y.~Qiu.
  Exchange graphs and Ext quivers.
  \emph{Adv. Math. }285 (2015), 1106--1154.
  (\arxiv{1109.2924})
%
\bibitem{KS}
  M.~Khovanov; P.~Seidel.
  Quivers, Floer cohomology, and braid group actions.
  \emph{J. Amer. Math. Soc. }15 (2002), no. 1, 203--271.
  (\arxiv{math/0006056})
%
\bibitem{LF}
  D.~Labardini-Fragoso.
  Quivers with potentials associated to triangulated surfaces.
  \emph{Proc. Lond. Math. Soc. (3) }98 (2009), no. 3, 797--839.
  (\arxiv{0803.1328})
%
\bibitem{QQ}
  Y.~Qiu.
  Decorated marked surfaces: spherical twists versus braid twists.
  \emph{Math. Ann. }365 (2016), no. 1-2, 595--633.
  (\arxiv{1407.0806})
%
\bibitem{QQ2}
  Y.~Qiu.
  Decorated marked surfaces (part B): topological realizations.
  \emph{Math. Z. } First Online: 23 March 2017.
  (Appendix of \arxiv{1407.0806})
%
\bibitem{QZ}
  Y.~Qiu; Y.~Zhou.
  Cluster categories for marked surfaces: punctured case.
  To appear in \emph{Compos. Math. }\arxiv{1311.0010}.
%
\bibitem{ST}
  P.~Seidel; R.~Thomas.
  Braid group actions on derived categories of coherent sheaves.
  \emph{Duke Math. J. }108 (2001), no. 1, 37--108.
  (\arxiv{math/0001043})
%
\bibitem{S}
  I.~Smith.
  Quiver algebras as Fukaya categories.
  \emph{Geom. Topol. }19 (2015), no. 5, 2557--2617.
  (\arxiv{1309.0452})
%
\bibitem{ZhZZ}
  J.~Zhang; Y.~Zhou; B.~Zhu.
  Cotorsion pairs in the cluster category of a marked surface.
  \emph{J. Algebra }391 (2013), 209--226.
  (\arxiv{1205.1504})
%
\end{thebibliography}
\end{document}